\documentclass[11pt,a4paper]{article}
\usepackage{geometry}

\newcommand{\footremember}[2]{%
    \footnote{#2}
    \newcounter{#1}
    \setcounter{#1}{\value{footnote}}%
}

\usepackage[pdftex]{graphicx}
\usepackage[latin1]{inputenc}
\usepackage[ruled,vlined]{algorithm2e}
\usepackage{epsfig,epstopdf}       
\usepackage{latexsym}
\usepackage{amsfonts,amssymb,amsthm,amsmath,amscd}
\usepackage{mathrsfs}
\usepackage{mathtools}
\usepackage{bbm}
\usepackage{seqsplit}
\usepackage{multirow}
\usepackage{caption}   
\usepackage{float} 
\usepackage{booktabs}
\usepackage{bigstrut}

\newcommand{\ds}{\displaystyle}
\newcommand{\T}{\mathsf{T}}

\def \rla{\rangle}
\def \lla{\langle}
\renewcommand{\Re}{{\rm I}\! {\rm R}}

\DeclareMathOperator*{\sgn}{sgn}
\DeclareMathOperator*{\diag}{diag}
\DeclareMathOperator*{\argmin}{arg\,min}
\DeclareMathOperator*{\rank}{rank}
\DeclareMathOperator*{\Ker}{Ker}
\DeclareMathOperator*{\Ran}{Ran}

\DeclareMathOperator*{\Fix}{Fix}
\DeclareMathOperator*{\supp}{supp}

\newtheorem{theorem}{Theorem}[section]
\newtheorem{lemma}[theorem]{Lemma}
\newtheorem{corollary}[theorem]{Corollary}
\newtheorem{proposition}[theorem]{Proposition}

\theoremstyle{definition}
\newtheorem{definition}[theorem]{Definition}
\newtheorem{example}[theorem]{Example}

\theoremstyle{remark}
\newtheorem{remark}[theorem]{Remark}

\numberwithin{equation}{section}

\usepackage[multiple]{footmisc}

\begin{document}

\title{Method of Alternating Projection for the Absolute Value Equation}

\author{%
 Jan Harold Alcantara\footremember{ntnu1}{Department of Mathematics,
                                         National Taiwan Normal University,
                                         Taipei 11677, Taiwan.} 
                     \footremember{dlsu}{Mathematics and Statistics Department,
                                         De La Salle University,
                                         Manila 1004, Philippines.
                                         Email:~\texttt{jan.alcantara@dlsu.edu.ph}}
 \and
 Jein-Shan Chen\footremember{ntnu2}{Department of Mathematics,
                                         National Taiwan Normal University,
                                         Taipei 11677, Taiwan.
                                         Email:~\texttt{jschen@math.ntnu.edu.tw}} 
 \and
 Matthew K.\ Tam\footremember{uom}{School of Mathematics and Statistics,
                                   The University of Melbourne,
                                   Parkville, VIC 3010, Australia.
                                   Email:~\texttt{matthew.tam@unimelb.edu.au}}
}

\maketitle

\begin{abstract}
	  A novel approach for solving the general absolute value equation $Ax+B|x| = c$ where $A,B\in \Re^{m\times n}$ and $c\in \Re^m$ is presented. We reformulate the equation as a feasibility problem which we solve via the method of alternating projections (MAP). The fixed points set of the alternating projections map is characterized under nondegeneracy conditions on $A$ and $B$. Furthermore, we prove linear convergence of the algorithm. Unlike most of the existing approaches in the literature, the algorithm presented here is capable of handling problems with $m\neq n$, both theoretically and numerically. 
\end{abstract}



\section{Introduction}
We consider an approach for dealing with the \emph{absolute value equation~(AVE)} given by 
\begin{equation}\label{eqn:AVE}
Ax+B|x| = c
\end{equation}
where $A\in \Re^{m\times n}$, $B\in \Re^{m\times n}$, $c\in \Re^m$ and $|x|$ denotes the componentwise absolute value of $x\in \Re^n$. Equation \eqref{eqn:AVE} with $m=n$ was first introduced by Rohn in~\cite{Rohn04} as a generalization of the equation $Ax-|x|=c$, the latter being the subject of numerous research works for almost two decades now; see~\cite{CQZ11,BCFP16,Haghani15,HH10,Man07-2,Man08,Man15,MM06,RHF14,ZW09}. Inspired by the work of Rohn~\cite{Rohn04}, Mangasarian introduced the general non-square system \eqref{eqn:AVE} in~\cite{Man07}. It is worth noting that interest in this equation is primarily motivated by its equivalence with the \emph{linear complementarity problem (LCP)}, which encompasses several mathematical programming problems~\cite{CD68,CPS92,Man07,MM06,Prokopyev09}. In addition, AVEs are also intimately related with mixed integer programming~\cite{Prokopyev09} and interval linear equations~\cite{Rohn89}. 

\medskip 
Due to its equivalence with the LCP, solving \eqref{eqn:AVE} is likewise an NP-hard problem~\cite{Man07}. Meanwhile, conditions for existence, non-existence and uniqueness of solutions of the AVE are reported in~\cite{HH10,MM06,RHF14,WL20}. On the numerical side, there are already many algorithms aimed at solving \eqref{eqn:AVE} in the special case when $m=n$ and $B=-I_n$. These algorithms can be roughly classified into four categories:
\begin{itemize}
	\item[(a)] \textit{Newton methods}. Most of the algorithms for solving AVE in the literature are based on modifications of the Newton method. For instance, an efficient semismooth Newton method is proposed in~\cite{Man08} to directly handle the nonsmooth equation \eqref{eqn:AVE} using the $B$-subdifferential of $|x|$ (see also Definition~\ref{defn:clarke} and equation \eqref{eqn:GNM}). Variants of the semismooth Newton method were also proposed, which include the inexact semismooth Newton method~\cite{BCFP16} and the generalized Traub's method~\cite{Haghani15}. Another approach followed by several works involves replacing the absolute value function by its smooth approximation, which then permits the use of the classical Newton method. This technique, known as the smoothing Newton method, was employed in several works such as in~\cite{CQZ11,SYC18}. A combination of both the semismooth and smoothing Newton method is also described in~\cite{ZW09}.
	\item[(b)] \textit{Picard iteration methods}. The Newton methods described above involve solving (exactly or approximately) linear systems of equations with different coefficient matrices at each iteration, which may be computationally expensive. On the other hand, in the Picard iteration method proposed in~\cite{RHF14}, a linear system with a fixed coefficient matrix $A$ is solved in each iteration (see also equation \eqref{eqn:Picard}), and thus may be more efficient than Newton methods. However, this approach is limited to the case when $A$ is invertible. A variant of this algorithm, known as the Picard-HSS iteration, is proposed in~\cite{Salkuyeh14} for handling the case that $A$ is non-Hermitian positive definite. The Douglas-Rachford splitting method recently proposed in~\cite{CYH21} may also be viewed as an extension of the Picard iterations \eqref{eqn:Picard}. 
	\item[(c)] \textit{Matrix splitting iteration method}. Under this category are two algorithms, namely the SOR-like iteration method~\cite{KM17} and the Gauss-Seidel iteration method~\cite{EHS17}. We note the observation that the Picard iteration method~\cite{RHF14} is a special case of the SOR-like iteration method, although the latter is derived from a matrix splitting approach. 
	\item[(d)] \textit{Concave minimization approach}. Mangasarian pioneered this approach by reformulating the AVE as a concave minimization problem and then using the successive linearization algorithm to solve the resulting reformulated problem~\cite{Man07-2,Man15}. In another recent work~\cite{AHM18}, the AVE is reformulated as a complementarity problem, which was smoothly approximated by a concave minimization problem. 
\end{itemize}
Meanwhile, to the best of our knowledge, the only method which can handle AVE \eqref{eqn:AVE} when $B\neq -I_n$ and $m\neq n$ is the successive linearization algorithm via concave minimization proposed in~\cite{Man07}. 

\medskip 
In this paper, we propose a simple approach for solving the general AVE \eqref{eqn:AVE} which, like~\cite{Man07}, does not require $B=-I_n$ or $m=n$. Moreover, our approach does not fall in any of the categories described above. We reformulate the AVE as a feasibility problem and then use the method of alternating projections (MAP) to solve the resulting problem. By introducing an auxiliary variable $y\in \Re^n$, we have that $x\in \Re^n$ solves \eqref{eqn:AVE} if and only if the pair $(x,y)\in \Re^n \times \Re^n$ solves 
\[Ax+By = c \qquad \text{and} \qquad y=|x| .\]
The above system of equations suggests the reformulation of \eqref{eqn:AVE} as a \emph{feasibility problem} given by 
\begin{equation}\label{eqn:AVE_as_a_FP}
\text{find}~(x,y)\in S_1 \cap S_2 \subseteq \Re^n \times \Re^n
\end{equation}
where the constraint sets, $S_1$ and $S_2$, are given by 
\begin{equation}\label{eqn:S1S2}
\begin{array}{rl}
S_1 & := \{ (x,y)\in \Re^n \times \Re^n :Ax+By = c\} \qquad \text{and} \\
\qquad S_2 & := \{ (x,y) \in \Re^n \times \Re^n: y = |x|\}.
\end{array}
\end{equation}
A simple algorithm to solve \eqref{eqn:AVE_as_a_FP} is the \emph{method of alternating projections (MAP)}: Given an initial point $z^0 = (x^0,y^0) \in \Re^n \times \Re^n$, MAP generates a sequence of iterates according to the rule:
\begin{equation}\label{eqn:MAP}
z^{k+1} = (x^{k+1},y^{k+1})\in (P_{S_1} \circ P_{S_2})(z^k) \qquad \forall k \in \mathbb{N},
\end{equation}
where $P_S$ is the possibly multivalued metric projector onto the set $S$ given by 
\[P_S(z) := \left\lbrace s\in S : \|s-z\| \leq \|t-z\| ~\forall t\in S\right\rbrace .\]
When $S$ is nonempty and closed, the image of $P_S$ at each point is nonempty. If in addition, $S$ is convex, the function $P_S$ is single-valued everywhere. Whenever $P_S(z)$ is single-valued, say $P_S(z)=\{s\}$, we simply write $s=P_S(z)$. 

\medskip 
For the theoretical analysis, our main objectives in this paper include the characterization of the set of fixed points of the alternating projections mapping $P_{S_1}\circ P_{S_2}$ and the establishment of the convergence of the MAP iterations \eqref{eqn:MAP}. Whenever the sequence produced by \eqref{eqn:MAP} is convergent, its limit may or may not be a solution of the feasibility problem \eqref{eqn:AVE_as_a_FP}. In Section \ref{sec:fixedpoints}, we provide sufficient conditions so that this limit corresponds to a point in $S_1\cap S_2$. A general condition is given for the case when $m$ and $n$ are arbitrary, but more specific conditions will be provided for the case $m=n$. For instance, one of our main results indicates that if the matrix $Q:=(A^{\T}+B^{\T})(A^{\T}-B^{\T})^{-1}$ is a $P$-matrix, then if the sequence generated by \eqref{eqn:MAP} is convergent, its limit solves the feasibility problem \eqref{eqn:AVE_as_a_FP}. In Section \ref{subsec:localconvergence}, we establish the local convergence of the algorithm \eqref{eqn:MAP} using the theory developed by Dao and Tam (2019) in~\cite{DT19}, which uses ideas originally developed in~\cite{BN14,Tam18}. We also  present a new complementarity function in Section \ref{subsec:newCfunction} which we use to provide an alternative convergence analysis of the MAP iterates. Despite the difficulty of proving the global convergence of \eqref{eqn:MAP} due to the nonconvexity of $S_2$, we prove in Section \ref{subsec:newCfunction} a weaker result implying the impossibility of the iterates to be trapped in some particular region not containing a point in $S_1\cap S_2$ (see Proposition~\ref{prop:impossible_to_be_trapped_in_wrong_region}). Moreover, by utilizing the convergence theory of Attouch, Bolte and Svaiter (2013) for semi-algebraic and tame problems~\cite{ABS13}, we prove in Section \ref{subsec:globalconvergence_relaxedMAP} the global convergence to stationary points of a relaxed version of the iterations \eqref{eqn:MAP} given by 
\begin{equation}\label{eqn:MAP_relaxation}
w^{k+1} \in (1-\gamma)P_{C_2}(w^k) + \gamma (P_{C_1}\circ P_{C_2})(w^k), \qquad \gamma \in (0,1).
\end{equation}
That is, we take the convex combination of the iterates \eqref{eqn:MAP} with the mapping $P_{C_2}$. Although MAP iterations \eqref{eqn:MAP} are not covered by the above relaxation, we note that the former is the limiting case of \eqref{eqn:MAP_relaxation} when $\gamma=1$. The linear rate of convergence of the MAP algorithm \eqref{eqn:MAP} to a point in $S_1\cap S_2$ is proved in Section \ref{subsec:rateofconvergence} using the theory of Lewis, Luke and Malick (2009) in~\cite{LLM09}. Finally, we present in Section \ref{subsec:relatedfixedpointalgorithm} another fixed point algorithm, which we call the MAP-LS algorithm, that can be derived from the method of alternating projections. 

\medskip 

The numerical contributions of our work are demonstrated in Section \ref{sec:numerical}. For the case $m=n$, the MAP and MAP-LS algorithms \eqref{eqn:MAP} are shown to be more robust and more efficient in solving randomly generated AVEs as compared with other methods from the four categories (a)-(d) of algorithms described above. For arbitrary $m$ and $n$, we illustrate the superior performance of MAP over the successive linearization algorithm in~\cite{Man07}, which is the only algorithm with which we can compare our method. Hence, our proposed algorithms have several merits from a numerical perspective, and are indeed an important contribution to the growing literature of AVE. 

\medskip 
In summary, the structure of this paper is as follows. In Section \ref{sec:fixedpoints}, we characterize the fixed point sets of the alternating projections mapping. Next, we present the convergence analysis of the algorithms in Section \ref{sec:convergenceanalysis}. Finally, we illustrate the applicability of our approach through numerical experiments in Section \ref{sec:numerical}.

\section{Fixed points of the alternating projections map}\label{sec:fixedpoints}
The method of alternating projections \eqref{eqn:MAP} is essentially aimed at finding a point $z\in \Re^n \times \Re^n$ that satisfies
\begin{equation}\label{eqn:fixedpointequation}
z\in (P_{S_1}\circ P_{S_2})(z) ,
\end{equation}
with the hope of solving the feasibility problem \eqref{eqn:AVE_as_a_FP}. A point satisfying \eqref{eqn:fixedpointequation} is called a \emph{fixed point} of the mapping $P_{S_1}\circ P_{S_2}$, and we denote the set of all fixed points of $P_{S_1}\circ P_{S_2}$ by $\Fix (P_{S_1}\circ P_{S_2})$. Necessarily, a point in the intersection of $S_1$ and $S_2$ must be a fixed point of $P_{S_1}\circ P_{S_2}$, that is, $S_1\cap S_2 \subseteq \Fix (P_{S_1}\circ P_{S_2})$. However, the converse is not necessarily true. Hence, if a sequence generated by MAP \eqref{eqn:MAP} is convergent, then its limit need only be a candidate solution to the feasibility problem \eqref{eqn:AVE_as_a_FP}.

\medskip 

This section is devoted to characterizing the set of fixed points of the alternating projections map. More precisely, we provide conditions on the matrices $A$ and $B$ which will allow us to determine which fixed points of $P_{S_1}\circ P_{S_2}$ belong to $S_1\cap S_2$. 

\subsection{Change of variables}\label{subsec:changeofvariables}

Instead of directly dealing with the sets $S_1$ and $S_2$ given by \eqref{eqn:S1S2}, we consider a change of variables which we find more convenient in our subsequent analysis. In general, we may consider any linear transformation $z = Rw$ where $z=(x,y)$, $w=(u,v)$ and $R\in \Re^{2n\times 2n}$ is a unitary matrix. Letting  $R = \begin{bmatrix}
R_1 & R_2 \\
R_3 & R_4
\end{bmatrix}$ with $R_i\in \Re^{n\times n}$, we see that $z\in S_i$ if and only if $w\in C_i$ for $i=1,2$ where 
\[C_1:= \{ w=(u,v)\in \Re^n\times \Re^n: \begin{bmatrix}
AR_1 + BR_3 & AR_2 + BR_4 
\end{bmatrix} w = c \}, \]
and 
\[C_2  := \{ w = (u,v)\in \Re^n\times \Re^n : |R_1u +R_2 v| = R_3 u +R_4 v\}. \]
With these, AVE \eqref{eqn:AVE} is also equivalent to the feasibility problem 
\begin{equation}\label{eqn:AVE_as_a_FP2}
\text{find}~w=(u,v)\in C_1 \cap C_2.
\end{equation}
Accordingly, we consider the MAP iterates given by 
\begin{equation}\label{eqn:MAP_2}
w^{k+1} \in (P_{C_1} \circ P_{C_2})(w^k).
\end{equation}	
Since we have chosen $R$ to be a unitary matrix, it follows that for all $z=(x,y) \in \Re^n \times \Re^n$, 
\[P_{S_i}(z) = RP_{C_i}(R^{\T}z), \quad  i=1,2. \]
Consequently, 
\begin{equation}\label{eqn:StoC}
(P_{S_1}\circ P_{S_2})(z) = R \left( (P_{C_1}\circ P_{C_2})(R^{\T}z)\right)
\end{equation}
Thus, if $z^0$ is the initial point for the original MAP iterates \eqref{eqn:MAP} and we set $w^0 =R^{\T}z^0$ for the iterations \eqref{eqn:MAP_2}, then $z^{k+1} = Rw^{k+1}$ for all $k\geq 0$. Moreover, we also have from \eqref{eqn:StoC} that
\begin{equation}\label{eqn:fixedpoints_transformed}
\Fix (P_{S_1} \circ P_{S_2}) = \{ Rw: w \in \Fix (P_{C_1} \circ P_{C_2})\} = R \Fix (P_{C_1} \circ P_{C_2}).
\end{equation} 

In the sequel, all our analyses and results are based on the constraint sets induced by $R =\frac{1}{\sqrt{2}}  \left[ \begin{array}{rr}
I_n & -I_n \\
I_n & I_n
\end{array} \right]$. Defining
\begin{equation}\label{eqn:T_map}
T := \begin{bmatrix}
A+B & -A+B
\end{bmatrix} \in \Re^{m \times 2n},
\end{equation}
it can be verified that $C_1$ is given by 
\begin{equation}\label{eqn:C1}
C_1 = \{ w \in \Re^n \times \Re^n ~:~ Tw = \sqrt{2}c\} .
\end{equation}
On the other hand, $(u,v) \in C_2$ if and only if $|u-v| = u+v$. Using the fact that $t+|t| = 2t_+$ where $t_+:= \max (0,t)$ (with the maximum understood in the pointwise sense), we see that $|u-v| = u+v$ if and only if $u-(u-v)_+ = 0$. Then, $C_2$ reduces to 
\begin{equation}\label{eqn:C2}
C_2 = \{ w=(u,v) \in \Re^n \times \Re^n ~:~ u \geq 0, ~v\geq 0 ,~\text{and}~\lla u, v \rla = 0 \}.
\end{equation}	
It follows that if $x$ solves \eqref{eqn:AVE}, then $(u,v)$ with $u=\frac{1}{\sqrt{2}}x_+$ and $v=\frac{1}{\sqrt{2}}(-x)_+$ solves the feasibility problem \eqref{eqn:AVE_as_a_FP2} with $C_1$ and $C_2$ given by \eqref{eqn:C1} and \eqref{eqn:C2}, respectively. Conversely, if $(u,v)$ solves \eqref{eqn:AVE_as_a_FP2}, then $x = \frac{1}{\sqrt{2}}(u-v)$ solves the AVE \eqref{eqn:AVE}.

\subsection{Projection formulas}\label{subsec:projectionformulas}
Important for our subsequent analysis and numerical simulations are the exact formulas for the projections involved in the MAP iterations given by \eqref{eqn:MAP_2}. The projection onto the affine set $C_1$ is well-known, which we recall in the following proposition. 

\medskip 
\begin{proposition}\cite[Lemma~4.1]{BK04}\label{prop:P_C1}
	Suppose that $c\in \Re^m$ is in the range of $T$ given by \eqref{eqn:T_map}. Then for any $w\in \Re^n$, we have 
	\begin{equation*}
	P_{C_1}(w) = w - T^{\dagger}(Tw-\sqrt{2}c),
	\end{equation*}
	where $T^{\dagger}$ is the Moore-Penrose inverse of $T$. 
\end{proposition}

\medskip

\medskip 
While $P_{C_1}$ is a single-valued operator, the projection onto $C_2$ is not due to the nonconvexity of $C_2$. 
\medskip 

\begin{proposition}\label{prop:P_C2}
	Let $w=(u,v)\in \Re^{n}\times \Re^n$ and let $z\in  P_{C_2}(w)$. Then for each $i=1,\dots, n$,
	\begin{equation}\label{eqn:P_C2}
	(z_i, z_{n+i}) \in \begin{cases}
	\{(0,(v_i)_+)\} & u_i < v_i \\
	\{((u_i)_+,0) \} & u_i > v_i \\
	\{ (0,(v_i)_+) , ((u_i)_+,0) \}  & u_i = v_i
	\end{cases}.
	\end{equation}
	In particular, $P_{C_2}$ is multivalued on $\{ (u,v): \exists i ~\text{such~that}~u_i = v_i>0\}$. 
\end{proposition}

\begin{proof}
	Fix $w=(u,v)\in \Re^n \times \Re^n$. To prove the result, we need to solve the minimization problem
	\[\min _{\bar{w}\in C_2} \|\bar{w} - w\|^2.\]
	Letting $\bar{w} = (\bar{u},\bar{v})\in \Re^n \times \Re^n$, we have
	\begin{eqnarray*}
		\|\bar{w} - w\|^2 & = & \sum _{i=1}^n (\bar{u}_i - u_i)^2 + \sum _{i=1}^n (\bar{v}_i - v_i)^2 \\ 
		& = & \sum _{i=1}^n \| (\bar{u}_i - u_i, \bar{v}_i-v_i) \|^2.
	\end{eqnarray*}
	As the last expression is separable, we only need to consider the projection of an arbitrary point $(s,t)\in \Re^2$ onto the set 
	\begin{equation}\label{eqn:C2bar}
	M := \{ (a,b):a\geq 0, ~b\geq 0 ~\text{and}~ab=0\},
	\end{equation}
	which can be easily calculated as 
	\begin{equation}\label{eqn:P_C2bar}
	P_{M}(s,t) = \begin{cases}
	\{ (0,t_+)\} & \text{if}~s<t \\
	\{(s_+,0)\} & \text{if}~s>t \\
	\{ (0,t_+), (s_+,0)\} & \text{if}~s=t
	\end{cases}
	\end{equation}	
	This gives the formula \eqref{eqn:P_C2}. 
\end{proof}

\medskip 
Note that because of the convexity of $C_1$, we know that the map $P_{C_1}$ is firmly nonexpansive, i.e., $\| P_{C_1}(w)- P_{C_1}(w')\| ^2\leq \lla w-w',P_{C_1}(w)-P_{C_1}(w') \rla $ for all $w,w'\in \Re^n \times \Re^n $. The same cannot be said for $P_{C_2}$ due to the nonconvexity of $C_2$. However, $P_{C_2}$ is firmly nonexpansive on some subsets of $\Re^{n}\times \Re^n$ as proved in Corollary~\ref{cor:PC2_nonexpansive}. Before we present this result, we first introduce some notations which will be used for the remaining parts of this paper. 	

\medskip 
We denote by $\mathscr{T}$ the collection of all functions $\tau: \{1,2,\dots, n\} \to \{ 1,2\}$, so that $|\mathscr{T}|=2^n$. For each $\tau \in \mathscr{T}$, we let $S_{\tau}$ denote the set of all $w=(u,v) \in \Re^n \times \Re^n$ such that for each $i=1,\dots,n$, we have $(u_i,v_i)\in K_j$ if $\tau(i)=j$ for $j=1,2$ where 
\begin{align}
K_1 & :=\{ (a,b)\in \Re^2 :a>b~\text{or}~a=b\leq 0\} \quad \text{and} \label{eqn:K1_a>b} \\ 
K_2 & :=\{ (a,b) \in \Re^2 : a<b~\text{or}~a=b\leq 0\}. \label{eqn:K2_a<b}
\end{align}
Observe that $$\bigcup_{\tau \in \mathscr{T}} S_{\tau} = \Re^n \times \Re^n \setminus \{ (u,v):u_i=v_i>0~\text{for~some~}i\}.$$ For each $\tau\in \mathscr{T}$, we also let $R_{\tau} := S_{\tau} \cap C_2$ so that $C_2 = \bigcup _{\tau \in \mathscr{T}} R_{\tau}.$ Thus, $w\in R_{\tau}$ if and only if for each $i=1,\dots,n$, we have $(u_i,v_i)\in M_j$ if $\tau (i)=j$ for $j=1,2$ where 
\begin{align}
M_1 & :=\{ (a,b)\in \Re^2 :a\geq 0 ~\text{and}~b=0\} \quad \text{and} \label{eqn:M1_a>b} \\ 
M_2 & :=\{ (a,b) \in \Re^2 : b\geq 0 ~\text{and}~a=0\}. \label{eqn:M2_a<b}
\end{align}

\medskip 
\begin{corollary}\label{cor:PC2_nonexpansive}
	$P_{C_2}$ is firmly nonexpansive on $S_{\tau}$ for any $\tau \in \mathscr{T}$.
\end{corollary}

\begin{proof}
	First, we note that the restriction of $P_{M}$ given by \eqref{eqn:P_C2bar} to $K_j$ is precisely the projection mapping $P_{M_j}$, where $M_j$ is given by \eqref{eqn:M1_a>b}-\eqref{eqn:M2_a<b}. Since $M_j$ is convex, then $P_{M_j}$ is firmly nonexpansive on $K_j$. It follows that $P_M$ is firmly nonexpansive on $K_j$. 
	
	Given $\tau \in \mathscr{T}$, take two points $w=(u,v)\in S_{\tau}$ and $w' = (u',v')\in S_{\tau}$. Then the points $(u_i,v_i)$ and $(u_i',v_i')$ both lie on $K_1$ or $K_2$ for each $i=1,\dots,n$. Then by firm nonexpansiveness of $P_{M}$, we obtain
	\begin{eqnarray}
	\|P_{C_2}(w) - P_{C_2}(w')\|^2 & = & \sum _{i=1}^n \| P_{M}(u_i,v_i) - P_{M}(u_i',v_i')\|^2 \notag \\
	& \leq & \sum _{i=1}^n \lla (u_i,v_i)-(u_i',v_i') , P_{M}(u_i,v_i) - P_{M}(u_i',v_i') \rla \notag  \\
	& = & \lla w-w',P_{C_2}(w)-P_{C_2}(w') \rla . \notag 
	\end{eqnarray}
	This proves the desired result. 
\end{proof}

\medskip 
By invoking the fact that $C_1$ given by \eqref{eqn:C1} is an affine set, the next proposition describes a property of the MAP iterates \eqref{eqn:MAP_2} which is based on the following observation: When $n=1$, the set $C_1$ defines a straight line in $\Re^2$ provided that $A$ and $B$ are not both zero. Intuitively, one can see that if the line $C_1$ intersects $C_2$ but does not pass through the origin, then $P_{C_1}(w) \notin \Re^2_-$ for any $w\in C_2$. For $n>1$, we may conjecture that if $\bar{w}:=P_{C_1}(w)$ with $w\in C_2$, we either have (i) $\bar{w}\notin \Re^{n}_-\times \Re^n_-$ or (ii) $(\bar{w}_i,\bar{w}_{n+i}) \notin \Re^2 _-$ for all $i=1,\dots, n$ for any $w\in C_2$. The following proposition indicates that (i) holds, and we illustrate in Example~\ref{example:quadrantofPC1PC2} that (ii) does not hold in general. 

\medskip 
\begin{proposition}\label{prop:w^knotinR2-}
	If $c\neq 0$, $C_1 \cap C_2 \neq \emptyset$ and $\{w^k\}_{k=0}^{\infty}$ is any sequence generated by \eqref{eqn:MAP_2}, then $w^k\notin  \Re^{n}_- \times \Re^n_-$ for all $k\geq 1$. 
\end{proposition}

\begin{proof}
	It is enough to show that given a point $w\in C_2$, we have $\bar{w}:=P_{C_1}(w) \notin \Re^n_- \times \Re^n_-$. Since  $C_1$ is a convex set, we have 
	\begin{equation*}
	\lla w - \bar{w} , w'-\bar{w} \rla \leq  0 \qquad \forall w'\in C_1.
	\end{equation*}
	In particular, we can take $w'=w^* \in C_1\cap C_2$ and obtain
	\begin{equation}\label{eqn:projection_inequality}
	\lla w - \bar{w}, w^* - \bar{w} \rla \leq 0.
	\end{equation}
	Since $c\neq 0$ and $T\bar{w}=\sqrt{2}c$, then $\bar{w}\neq 0$. Thus, if $\bar{w} \in \Re^n_- \times \Re^n_-$, then there exists some $i\in \{1,2\dots, 2n\}$ such that $\bar{w}_i< 0$. Since $w,w^*\geq 0$, we must have $w_i - \bar{w}_i >0$ and $w_i^*-\bar{w}_i>0$. Meanwhile, we also have that $w_j-\bar{w}_j \geq 0$ and $w_j^*-\bar{w}_j \geq 0$ for all $j$. In turn, we will obtain $\lla w-\bar{w}, w^*-\bar{w}\rla >0$ which contradicts \eqref{eqn:projection_inequality}. Hence, $\bar{w}\notin \Re^n _- \times \Re^n _-$ as desired.
\end{proof}

\medskip 
\begin{example}\label{example:quadrantofPC1PC2}
	Let $A= \left[ \begin{array}{rr}
	3 & -8 \\ 3 & 0
	\end{array} \right]$, $B=-I_2$ and $c=(6,9)/\sqrt{2}$. It can be verified that $C_1\cap C_2 = \{(3/\sqrt{2},0,0,0) \}$. For $w=(0,0,1,0) \in C_2$, one can check that $\bar{w}:=P_{C_1}(w) \approx  (1.8042,-0.5569,-0.7921,-0.6540)$. Note that $(\bar{w}_2,\bar{w}_4 ) \in \Re^2_-$. 
\end{example}

\medskip

\subsection{Characterization of fixed points for arbitrary $m$ and $n$}

We now provide a general condition which will allow us to distinguish which fixed points of $P_{C_1}\circ P_{C_2}$ belong to $C_1\cap C_2$. In the following, we denote by $\Ker(T)$ and $\Ran (T)$ the kernel and range of $T$, respectively. Given any affine set $S\subseteq \Re^n$, we denote its orthogonal complement by $S^{\perp}$. 

\begin{theorem}[Characterization of fixed point sets for arbitrary $m$, $n$]\label{theorem:fixedpoints_arbitrary} 
	Let $T\in \Re^{m\times 2n}$ be given by \eqref{eqn:T_map} and suppose that 
	\begin{equation}\label{eqn:linearlyregularintersection}
	\Ker(T)^{\perp} \cap \hat{C}_2 = \{0\},
	\end{equation}
	where 
	\begin{equation}\label{eqn:C2hat}
	\hat{C}_2 :=\{ w=(u,v) \in \Re^n \times \Re^n: u_iv_i =0 ~\forall i =1,\dots , n\}.
	\end{equation}
	Denote 
	\begin{eqnarray}
	\Omega & := & \{ w = (u,v) \in \Re^n \times \Re^n : (u_i,v_i) \notin \Re^2_{--} ~\forall i=1,\dots, n \}. \notag 
	\end{eqnarray}
	Then for any $c  \in \Re^m$,
	\begin{equation*}\label{eqn:fixedpoints=C1capC2_arbitrary}
	\Fix (P_{C_1} \circ P_{C_2}) \cap \Omega = C_1 \cap C_2 .
	\end{equation*}
\end{theorem}

\begin{proof}
	We note first that if $c\notin \Ran (T)$, then $C_1 = \emptyset$. Since $\Fix (P_{C_1}\circ P_{C_2}) \subseteq C_1$, then $\Fix (P_{C_1}\circ P_{C_2}) = \emptyset$. Hence, the result necessarily holds. Suppose now that $c\in \Ran (T)$ so that $C_1 \neq \emptyset$. Since $C_1\cap C_2 \subseteq \Fix (P_{C_1}\circ P_{C_2})$ and $C_2 \subseteq \Omega$, then $C_1\cap C_2 \subseteq \Fix (P_{C_1}\circ P_{C_2}) \cap \Omega$. To prove the other inclusion, suppose that $w=(u,v) \in \Fix (P_{C_1}\circ P_{C_2}) \cap \Omega$. Since $w \in (P_{C_1}\circ P_{C_2})(w)$, then $w=P_{C_1}(w')$ for some $w'\in P_{C_2}(w)$. Since $C_1$ is an affine set, it follows that $w-w' \in \Ker(T)^{\perp}$. 
	
	We also have that $w\in \Omega$ so that we may partition its components using the following index sets:
	\begin{eqnarray}
	I& :=&  \{ i\in \{1,2,\dots ,n\}: u_i > v_i ~\text{and}~u_i\geq 0\} \notag \\ 
	J& :=&  \{ i\in \{1,2,\dots ,n\}: u_i = v_i \geq 0 \} \notag \\ 
	K& :=&  \{ i\in \{1,2,\dots ,n\}: u_i < v_i ~\text{and}~v_i\geq 0\} \notag 
	\end{eqnarray}
	By rearranging the columns of $A$ and $B$ if necessary, we may suppose that $u=(u_I,u_J,u_K) \in \Re^n$ where $u_{\Lambda}$ denotes the components of $u$ indexed by $\Lambda  \in \{ I, J, K\}$. Accordingly, we let $v=(v_I,v_J,v_K)\in\Re^n$. Consequently, we have from Proposition~\ref{prop:P_C2} that $w' = (u',v')$ where $u'=(u_I, u_J', 0_{|K|})$ and $v'=(0_{|I|}, v_J', v_K)$ with $(u_j',v_j')\in \{(u_j,0), (0,v_j)\}$ . Then $(w-w')_i(w-w')_{n+i}=0$ for all $i=1,\dots, n$, that is, we have $w-w' \in \hat{C}_2$.
	
	To summarize, we have shown that $w-w'\in \Ker(T)^{\perp}\cap \hat{C}_2$. By \eqref{eqn:linearlyregularintersection}, it follows that $w=w'$. Hence, $w\in C_2$. This completes the proof.  
\end{proof}

\medskip
The condition \eqref{eqn:linearlyregularintersection} is not easy to verify for the case $m\neq n$. In the next subsection where we discuss the case $m=n$, we use the notion of nondegenerate matrices to provide an easier-to-verify condition on $A$ and $B$ that will result to a map $T$ that satisfies \eqref{eqn:linearlyregularintersection}.
\medskip 

Moreover, we also note the following observation: From Proposition~\ref{prop:w^knotinR2-}, we see that $\Fix (P_{C_1}\circ P_{C_2}) \subseteq (\Re^n \times \Re^n) \setminus (\Re^n_{--} \times \Re^n_{--})$ and note that 
\[\Omega \subseteq (\Re^n \times \Re^n) \setminus (\Re^n_{--} \times \Re^n_{--}) .\]
Hence, it is not necessary that $\Fix (P_{C_1}\circ P_{C_2}) \subseteq \Omega$. In Example~\ref{example:fixedpoints}, we illustrate the importance of intersecting the set of fixed points with the set $\Omega$. For the case $m=n$, we also provide a sufficient condition so that the set of fixed points is necessarily contained in $\Omega$ (see Theorem~\ref{theorem:fixedpoints_Pmatrix}), which in turn implies that $\Fix (P_{C_1}\circ P_{C_2})$ is equal to $C_1\cap C_2$.
\medskip

\subsection{Characterization of fixed points for $m=n$}\label{subsec:fixedpoints_m=n}
To prove our main result for the case $m=n$, we first establish some key lemmas which are prerequisites to obtaining conditions on $A$ and $B$ that will imply \eqref{eqn:linearlyregularintersection}. 

\medskip 
\begin{lemma}\label{lemma:pseudoinverse_property}
	Let $T\in \Re^{m\times 2n}$ be given by \eqref{eqn:T_map}. Suppose that at least one among the matrices $A$, $B$, $A+B$ and $A-B$ is of full row rank. Then $\rank (T)=m$. In particular, $\rank (T)=n$ when $m=n$ and $B=-I_n$.
\end{lemma}

\begin{proof}
	If either $A+B$ or $A-B$ has rank $m$, then it is clear that $\rank (T)=m$. For the other cases when either $A$ or $B$ is of full row rank, we note that $\rank (T)=\rank (TT^\T)=\rank (2(AA^\T+BB^\T))$. The matrices $AA^\T$ and $BB^\T$  are positive semidefinite, and at least one of them is positive definite by our full rank assumption. Hence, the claim of the lemma follows.
\end{proof}

\medskip 
The next lemma precisely describes the elements of the set $\Ker (T)^{\perp}$. 

\medskip 
\begin{lemma}\label{lemma:property_Q}
	Suppose $m=n$ and $T\in \Re^{n\times 2n}$ is given by \eqref{eqn:T_map}, and suppose that $A-B$ is nonsingular. Then $\Ker (T)^{\perp} = \Ker(\begin{bmatrix}
	I_n & Q
	\end{bmatrix})$ where 	
	
	\begin{equation}\label{eqn:Qmatrix}
	Q:=(A^\T+B^\T)(A^\T-B^\T)^{-1}
	\end{equation}
\end{lemma}

\begin{proof}
	Let $w=(u,v)\in \Ker(T)^{\perp}= \Ran (T^{\T})$. Then there exists $x\in \Re^n$ such that $T^{\T}x = w$. It follows that $u=(A^{\T}+B^{\T})x$ and $v=-(A^{\T}-B^{\T})x$. By invertibility of $A-B$, we see that 
	\[u = (A^{\T}+B^{\T})x = -(A^{\T}+B^{\T})(A^{\T}-B^{\T})^{-1}v = -Qv.\]
	Hence, we have $\Ker (T)^{\perp} \subseteq  \Ker(\begin{bmatrix}
	I_n & Q
	\end{bmatrix})$. Meanwhile, we also have that $$\dim (\Ker(\begin{bmatrix}
	I_n & Q
	\end{bmatrix})) = 2n - \rank (\begin{bmatrix}
	I_n & Q
	\end{bmatrix}) = n $$ by the rank-nullity theorem, and $$\dim (\Ker (T)^{\perp}) = \rank (T^{\T}) = \rank (T) = n,$$ by Lemma~\ref{lemma:pseudoinverse_property}. Thus, $$\dim (\Ker(\begin{bmatrix}
	I_n & Q
	\end{bmatrix})) = \dim (\Ker (T)^{\perp}) .$$ With these, we conclude that $\Ker (T)^{\perp} = \Ker(\begin{bmatrix}
	I_n & Q
	\end{bmatrix})$. 
\end{proof}

\medskip 
Note that we can derive a result similar to Lemma~\ref{lemma:property_Q} if we rather assume that $A+B$ is nonsingular. 

\medskip 
Now that we have described the set $\Ker (T)^{\perp}$, we next focus on finding conditions which will imply \eqref{eqn:linearlyregularintersection}. Nondegenerate matrices, as defined below, will play a major role in our analysis. 

\begin{definition}\cite{CPS92}\label{defn:nondegenerate}
	A matrix $Q\in \Re^{n\times n}$ is nondegenerate if all its principal minors are nonzero, i.e. the principal submatrix $Q_{\Lambda \Lambda}$ is nonsingular for all $\Lambda \subseteq \{1,\dots,n\}$. We call $Q$ degenerate if it is not a nondegenerate matrix. 
\end{definition}

\medskip 

\begin{lemma}\label{lemma:switchingcolumns_nonsingular}
	Let $A,Q\in \Re^{n\times n}$ be nonsingular matrices where $Q$ is a nondegenerate matrix, and let $B:= AQ$. Let $\Lambda\subseteq \{ 1,\dots, n\}$ and let $A'$ be the $n\times n$ matrix obtained by replacing the columns of $A$ indexed by $\Lambda$ by those columns of $B$ indexed by $\Lambda$. Then $A'$ is nonsingular.
\end{lemma}

\begin{proof}
	Without loss of generality, assume that $\Lambda = \{ 1,2,\dots, k\}$ with $k\leq n$. Denote the columns of $A$ by $\{ v_1,v_2,\dots, v_n\}$ and the columns of $B$ by $\{ \bar{v}_1,\bar{v}_2,\dots, \bar{v}_n\}$. To prove the claim, we only need to show that $\{ \bar{v}_1,\bar{v}_2, \dots, \bar{v}_k, v_{k+1},\dots,v_n\}$ is linearly independent. 
	
	\medskip 
	
	Suppose that $\ds \sum _{j=1}^{k}a_j \bar{v}_j + \sum _{j=k+1}^n a_j v_j=0$ for some constants $a_1,\dots, a_n$. By the definition of $B$, we have $\ds \bar{v}_j = \sum _{i=1}^n q_{ij}v_i$ for all $j$, where $q_{ij}$ is the $(i,j)$-entry of $Q$. Direct computations lead us to
	\begin{eqnarray*}
		& & \left(\sum _{j=1}^k a_j q_{1j}\right) v_1 + \cdots + \left(\sum _{j=1}^k a_j q_{kj}\right) v_k \notag  \\ 
		& & + \left( a_{k+1}+\sum _{j=1}^k a_j q_{k+1,j} \right) v_{k+1} + \cdots +  \left( a_{n}+\sum _{j=1}^k a_j q_{nj} \right) v_{n} = 0. \label{eqn:lin_indep}
	\end{eqnarray*} 
	Since the $v_i$'s are linearly independent, all the coefficients above should be equal to zero. From the first $k$ terms, we obtain that $Q_{\Lambda \Lambda}(a_1,\dots, a_k)^\T=0$. Since $Q_{\Lambda \Lambda }$ is nonsingular by nondegeneracy of $Q$, then $a_j=0$ for all $j=1,\dots, k$ which consequently gives $a_j=0$ for all $j>k$. 
\end{proof}

\begin{proposition}\label{lemma:Pnondegenerate_relationwithD}
	Let $m=n$ and suppose that the matrix $Q$ defined by \eqref{eqn:Qmatrix} is nondegenerate. Let $\Lambda_1\subseteq \{1,\dots, n\}$ and $\Lambda_2=\{ n+i : i \notin \Lambda_1 \}$. Then the columns of $[I_n~~Q]$ indexed by $\Lambda_1\cup \Lambda_2$ are linearly independent. Consequently, condition \eqref{eqn:linearlyregularintersection} holds. 
\end{proposition}

\begin{proof}
	Set$A=I_n$ and $\Lambda = \{ 1,\dots , n \}\setminus \Lambda_1$. Then the columns of matrix $A'$ described in Lemma~\ref{lemma:switchingcolumns_nonsingular} are precisely the columns of $D:=[I_n~~Q]$ indexed by $\Lambda_1 \cup \Lambda_2$. Consequently, $A'$ is nonsingular and so the first claim of the proposition holds. 
	
	If $w=(u,v)\in \Ker (T)^{\perp} \cap \hat{C}_2$, then since $\Ker (T)^{\perp}  = \Ker (D)$ by Lemma~\ref{lemma:property_Q}, we have 
	\begin{equation}\label{eqn:lincombi}
	0 = Dw = \sum_{i\in \Lambda_1} u_i d_i + \sum_{i\in \Lambda_2'} v_i d_i
	\end{equation}
	where $d_i\in \Re^n$ is the $i$th column of $D$, $\Lambda_1 := \{i:u_i\neq 0 \}$ and $\Lambda_2' := \{i:v_i\neq 0\}$. In other words, the right-hand side of \eqref{eqn:lincombi} is a linear combination of the columns of $D$ indexed by $$\Lambda:= \Lambda_1 \cup \{n+i:i\in \Lambda_2'\} \subseteq \Lambda_1 \cup \Lambda_2.$$
	Thus, the columns indexed by $\Lambda$ must be linearly independent, i.e. $\Lambda_1 = \Lambda_2' = \emptyset$ so that $w=0$. 
\end{proof}
\medskip 

As an immediate consequence of the above result and Theorem~\ref{theorem:fixedpoints_arbitrary}, we have the following. 

\begin{theorem}[Characterization of fixed point sets for $m=n$]\label{theorem:fixedpoints} 
	Let $m=n$. Suppose that $Q$ given by \eqref{eqn:Qmatrix} is a nondegenerate matrix and $\Omega$ is as defined in Theorem~\ref{theorem:fixedpoints_arbitrary}. Then for any $c \in \Re^n$,
	\begin{equation*}\label{eqn:fixedpoints=C1capC2}
	\Fix (P_{C_1} \circ P_{C_2}) \cap \Omega = C_1 \cap C_2.
	\end{equation*}
\end{theorem}

\begin{proof}
	Since $Q$ is nondegenerate, we have from Proposition~\ref{lemma:Pnondegenerate_relationwithD} that condition \eqref{eqn:linearlyregularintersection} holds. Hence, the claim follows from Theorem~\ref{theorem:fixedpoints_arbitrary}. 
\end{proof}

\medskip 

\begin{remark}\label{remark:howtogetnondegeneracy} One can guarantee that the matrix $Q$ given by \eqref{eqn:Qmatrix} is nondegenerate if $\sigma_{\min} (A) > \sigma_{\max}(B)$, i.e. the smallest singular value of $A$ is greater than the largest singular value of $B$. To see this, note that for all $x\in \Re^n$, 
	\begin{eqnarray}
	x^\T (A^\T+B^\T)(A^\T-B^\T)^{-1}x & = & y^\T (A-B)(A^\T + B^\T)y , ~ \text{where}~y=(A^\T - B^\T)^{-1}x \notag \\
	& = &   y^\T(AA^\T - BA^\T + AB^\T - BB^\T)y \notag \\
	& = & y^\T (AA^\T-BB^\T)y \notag  \\
	& \geq & (\lambda _{\min}(AA^\T) - \lambda_{\max}(BB^\T)) \|y\|^2 \notag \\
	& = & (\sigma_{\min}(A) - \sigma_{\max}(B))\|y\|^2 , \notag 
	\end{eqnarray}
	where the third equality follows from $y^\T BA^\T y = y^\T A B^\T y$. It follows that $Q$ is a positive definite matrix, which is necessarily nondegenerate. By a similar computation, the condition $\sigma_{\max}(A) < \sigma_{\min}(B)$ implies nondegeneracy of $Q$. $_{\blacksquare }$ 
\end{remark}

\medskip 
The following example demonstrates the importance of nondegeneracy of $Q$ as well as the significance of intersecting the set of fixed points with $\Omega$.
\begin{example}\label{example:fixedpoints} \text{} 
	\begin{enumerate}
		\item Let $A= \left[ \begin{array}{rr}
		1 & 2 \\ 3 & 4 
		\end{array} \right]$, $B=-I_2$ and $c=(-10,-19)/\sqrt{2}$. Then $Q = \left[ \begin{array}{rr}
		-1.5 & 1.5 \\ 1 & 0 
		\end{array} \right]$, which is a degenerate matrix. Moreover, it can be verified that $w = (-0.9231, 4.7026,$ $9.0872, 0.6154) \in \Fix (P_{C_1} \circ P_{C_2}) \cap \Omega $. Clearly, however, $w\notin C_1 \cap C_2$. We note that the problem is feasible, i.e. $C_1\cap C_2 \neq \emptyset$. For instance, both $(0,0,3,2)/\sqrt{2}$ and $(2 ,0 ,0, 5)/\sqrt{2}$ are solutions of the feasibility problem.
		\item Let $A=1/2$, $B=3/2$ and $c=-\sqrt{2}$. Then $Q=-2$, which is nondegenerate. Moreover, $C_1\cap C_2 = \emptyset$ and $\Fix (P_{C_1} \circ P_{C_2})  = \{ (-0.8,-0.4)\}$ so that $C_1\cap C_2 \neq \Fix (P_{C_1} \circ P_{C_2})$. Nevertheless, we see that $\Fix (P_{C_1} \circ P_{C_2}) \cap \Omega = C_1 \cap C_2 $. 
	\end{enumerate}
\end{example}

\medskip 
It turns out that the signs of the principal minors of $Q$ play an important role in characterizing the fixed points of the alternating projections map. In particular, the fixed points are necessarily contained in the set $\Omega$  defined in Theorem~\ref{theorem:fixedpoints} if all the principal minors of $Q$ are positive. Such a matrix is called a $P$-matrix~\cite{CPS92}. In contrast, we see from Example~\ref{example:fixedpoints}.2 that intersecting the set of fixed points with $\Omega$ is necessary if there exists a negative principal minor. 

\medskip 
To characterize the set of fixed points for the $P$-matrix case, we need the following lemma.

\medskip 
\begin{lemma}\cite[Theorem~3.3.4]{CPS92}\label{lemma:Pmatrix}
	$Q\in \Re^{n\times n}$ is a $P$-matrix (i.e., all of its principal minors are positive) if and only if whenever $x_i(Qx)_i\leq 0$ for all $i=1,\dots, n$, we have $x=0$.
\end{lemma}

\medskip 

\begin{theorem}\label{theorem:fixedpoints_Pmatrix}
	Let $m=n$. Suppose that $Q$ given by \eqref{eqn:Qmatrix} is a $P$-matrix. Then for any $c\in \Re^n$, we have 
	\[\Fix (P_{C_1}\circ P_{C_2}) = C_1\cap C_2. 
	\]
	In particular, by Remark~\ref{remark:howtogetnondegeneracy}, the above equality holds if $\sigma_{\min} (A) > \sigma_{\max} (B)$. 
\end{theorem}

\begin{proof}
	Suppose that $w\in (P_{C_1}\circ P_{C_2})(w)$. As in the proof of Theorem~\ref{theorem:fixedpoints_arbitrary}, we have $w-w' \in \Ker (T)^{\perp}$ where $w' \in P_{C_2}(w)$. Since $Q$ is nondegenerate, $A-B$ is necessarily nonsingular so that by Lemma~\ref{lemma:property_Q}, $w-w' \in \Ker ([I_n~~Q])$, i.e. 
	\begin{equation}\label{eqn:fixedpointequation2}
	u-u' + Q(v-v') = 0 .
	\end{equation}
	Observe that to prove the desired result, Theorem~\ref{theorem:fixedpoints} implies that it is enough to prove that $w\in \Omega$, i.e. $(u_i,v_i)\notin \Re^2_{--}$ for all $i$. Suppose to the contrary that there exists an index $j$ such that $u_j,v_j<0$. Then from Proposition~\ref{prop:P_C2}, we know that $u_j'=v_j'=0$ so that $v_j-v_j'<0$. In particular, $v-v'$ is a nonzero vector. Consequently, by Lemma~\ref{lemma:Pmatrix}, there exists some $l$ such that
	\begin{equation}\label{eqn:Pmatrixinequality}
	(v_l-v_l')(Q(v-v'))_l >0
	\end{equation}
	We consider two cases:
	\begin{itemize}
		\item[(i)] Suppose that $v_l-v_l'>0$. From Proposition~\ref{prop:P_C2}, this can only happen if $u_l\geq v_l$ and $(u_l',v_l')=(u_l,0)$. Thus, we obtain that $u_l-u_l'=0$. From equation \eqref{eqn:fixedpointequation2}, it follows that $(Q(v-v'))_l=0$. This is a contradiction to \eqref{eqn:Pmatrixinequality}. 
		\item[(ii)] Suppose that $v_l-v_l'<0$. We conclude from Proposition~\ref{prop:P_C2} that $v_l<0$ and $v_l'=0$. Moreover, it also follows from the same proposition that $u_l-u_l'\leq 0$. From equation \eqref{eqn:fixedpointequation2}, it must be the case that $(Q(v-v'))_l\geq 0$. Hence, $(v_l-v_l')(Q(v-v'))_l \leq 0$. However, this is a direct contradiction to \eqref{eqn:Pmatrixinequality}.
	\end{itemize}
	Hence, it is impossible that there exists $j$ such that $(u_j,v_j)\in \Re^2_{--}$, i.e. $w\in \Omega$. This completes the proof. 
\end{proof}

\medskip 
\begin{remark}\label{remark:AVE_equiv_LCP}
	If $A-B$ is nonsingular, then the feasibility problem \eqref{eqn:AVE_as_a_FP2} is equivalent to solving the system
	\[u\geq 0, \qquad F(u):= Q^{\T}u - \sqrt{2}(A-B)^{-1}c \geq 0, \qquad \text{and} \qquad \lla u, F(u) \rla = 0, \]
	known in the literature as a \emph{linear complementarity problem} (LCP). The above LCP has a unique solution for all $c\in \Re^n$ if and only if $Q$ is a $P$-matrix~\cite{CPS92}. Thus, Theorem~\ref{theorem:fixedpoints_Pmatrix} indicates that if $Q$ is a $P$-matrix, then for any $c\in \Re^n$, the set of fixed points of $P_{C_1}\circ P_{C_2}$ consists of a single point, which is precisely the solution of the feasibility problem \eqref{eqn:AVE_as_a_FP2}. $_{\blacksquare }$ 
\end{remark}

\medskip 

The next result, which is a very special case, provides another condition for the equality of the set of fixed points and the intersection of $C_1$ and $C_2$. 
\medskip 

\begin{theorem}\label{theorem:fixedpoints_2n}
	Suppose that $C_1 \cap R_{\tau}\neq \emptyset$ for all $\tau \in \mathscr{T}$. Then 
	\[\Fix (P_{C_1} \circ P_{C_2}) = C_1 \cap C_2  \]
\end{theorem}

\begin{proof}
	Suppose $w = P_{C_1}(w')$ where $w' \in P_{C_2}(w)$. Choose $\tau \in \mathscr{T}$ such that $w' \in R_{\tau}$, so that $w' = P_{R_{\tau }}(w)$. Taking $w^* \in C_1 \cap R_{\tau}$ and using the convexity of $C_1$ and $R_{\tau}$, we obtain $\lla w' - w, w^* - w\rla \leq 0$ and $\lla w-w',w^*-w' \rla \leq 0$, respectively. Adding these two inequalities, we see that $\| w' - w \|^2 \leq 0$ so that $w=w'$ and therefore $w\in C_1 \cap C_2$. The other inclusion is trivial, and thus, the proof is complete.  
\end{proof}
\medskip 
As a consequence, we state the following corollary whose hypothesis is the setting considered in~\cite{MM06}.

\begin{corollary}
	Let $A\in \Re^{n\times n}$, $B=-I_n$ and $c<0$. If $\|A\|_{\infty} < \frac{\alpha}{2}$ where $\alpha = \frac{\min _i |c_i|}{\max _i |c_i|}$, then 
	\[\Fix (P_{C_1} \circ P_{C_2}) = C_1 \cap C_2   \]
\end{corollary}

\begin{proof}
	From~\cite[Proposition~6]{MM06}, we know that the AVE \eqref{eqn:AVE} has exactly $2^n$ distinct solutions, each of which has no zero components and has different sign pattern. Thus, each $R_{\tau}$ contains a point in $C_1\cap C_2$ in its interior. The claim then follows from Theorem~\ref{theorem:fixedpoints_2n}.
\end{proof}

\section{Convergence analysis}\label{sec:convergenceanalysis}
In this section, we discuss the convergence issues related to the proposed method of alternating projections. In Section \ref{subsec:localconvergence}, we present some local convergence results which are direct consequences of the theory developed in \cite{DT19}. We present an alternative local convergence analysis in Section \ref{subsec:newCfunction} through the use of a new complementarity function. A by-product of this alternative analysis is the global convergence of MAP for homogeneous AVE. In addition, we also prove in Section \ref{subsec:newCfunction} that under a nondegeneracy assumption, the MAP iterates cannot be trapped in some region $S_{\tau}$ (defined in Section \ref{subsec:projectionformulas}) if $S_{\tau}$ does not contain a solution of the feasibility problem \eqref{eqn:AVE_as_a_FP2}. In Section \ref{subsec:rateofconvergence}, we establish the linear rate of convergence of MAP. A globally convergent relaxation of MAP is presented in \ref{subsec:globalconvergence_relaxedMAP}. Finally, another algorithm derived from the fixed point relation $w\in (P_{C_1}\circ P_{C_2})(w)$ is described in Section \ref{subsec:relatedfixedpointalgorithm}. 

\subsection{Convergence of MAP}\label{subsec:localconvergence}
The method of alternating projections and its generalization to more than two sets are globally convergent when the involved sets are convex~\cite{Breg65}. For our problem \eqref{eqn:AVE_as_a_FP2}, the set $C_1$ is affine (hence, convex) while $C_2$ is a nonconvex set. Nevertheless, $C_2$ is a \emph{union convex set}, i.e. it can be expressed as a finite union of closed convex sets~\cite{DT19}. In particular, we can write $C_2$ as $ \ds C_2 = \bigcup _{\tau \in \mathscr{T}} R_{\tau}$
where each $R_{\tau}$ is a closed convex set as defined in the preceding section. Thus, the local convergence of MAP is a direct consequence of~\cite[Corollary~6.2]{DT19}.

\begin{theorem}[Local convergence of MAP]\label{theorem:local_general}
	Suppose $w^* \in C_1\cap C_2$. Then there exists sufficiently small $\delta>0$ such that for any $w^0$ with $\|w^0 - w^*\|<\delta$, any sequence generated by \eqref{eqn:MAP_2} converges to a solution of \eqref{eqn:AVE_as_a_FP2}. 	
\end{theorem}

\medskip 
On the other hand, the global convergence of MAP to solutions of the feasibility problem \eqref{eqn:AVE_as_a_FP2} is not always guaranteed. 

\begin{example}
	In Example~\ref{example:fixedpoints}.1, the MAP iterates \eqref{eqn:MAP_2} may converge to a fixed point of $P_{C_1}\circ P_{C_2}$ that does not belong to $C_1\cap C_2$. For instance, if we set $w^0 = (-1, 5, 9, 1)$, it can be verified that $w^k$ converges to the point $w=(-0.9231,4.8077,9.1923,0.6154) \in \Fix (P_{C_1}\circ P_{C_2})\setminus (C_1\cap C_2)$.
\end{example} 

In fact, the following example shows that unique solvability does not imply global convergence.

\begin{example}
	Let $A=1$, $B=-1$ and $c=-2/\sqrt{2}$. Then $C_1$ is the horizontal line $v=1$ while $C_2$ is the union of the nonnegative $u$ and $v$ axes. In Figure~\ref{fig:MAPdivergent}, we see that $w^*=(0,1)$ is the unique solution to \eqref{eqn:AVE_as_a_FP2}. Meanwhile, MAP is not globally convergent to $w^*$. 
\end{example}

\begin{figure}[hbtp!]
	\begin{center}
		\includegraphics[scale=1.5]{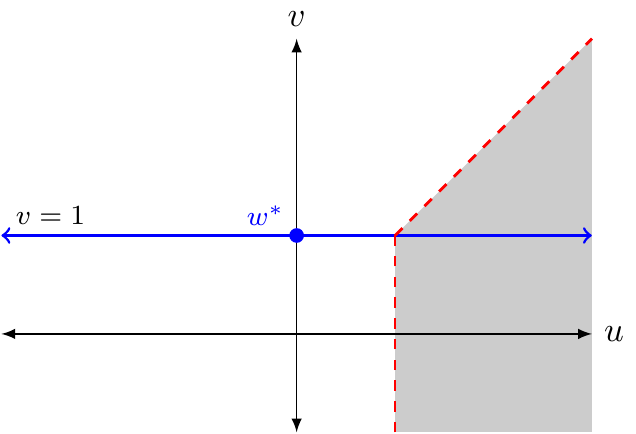}
	\end{center}
	\caption{The method of alternating projections converges to a point in $\Fix (P_{C_1}\circ P_{C_2})\setminus (C_1\cap C_2)$ when the initial point $w^0$ lies on the gray region, while the convergence to $w^*$ depends on the selected element of $P_{C_2}(w^0)$ if $w^0$ lies on the red dashed line.}\label{fig:MAPdivergent}
\end{figure}

\medskip 
In both of the examples above, we note that the matrix $Q$ defined by \eqref{eqn:Qmatrix} is degenerate. This suggests that for the case $m=n$, nondegeneracy of $Q$ may be a necessary condition for global convergence to $C_1\cap C_2$. We leave this as a conjecture which is worth further investigation. Note, however, that nondegeneracy is not sufficient for global convergence to solutions (for example, see Example~\ref{example:fixedpoints}.2). 

\medskip
We close this section by identifying two specific instances when the method of alternating projections is globally convergent. 

\begin{proposition}
	Suppose $T\in \Re^{m\times 2n}$ has full column rank. Then the feasibility problem \eqref{eqn:AVE_as_a_FP2} has a solution if and only if $TT^{\dagger}c=c$ and $\sqrt{2}T^{\dagger}c\in C_2$. In particular, $\sqrt{2}T^{\dagger}c$ is the unique solution to \eqref{eqn:AVE_as_a_FP2} whenever a solution exists. Moreover, any sequence generated by \eqref{eqn:MAP_2} converges finitely to $\sqrt{2}T^{\dagger}c$ (after one iteration). 
\end{proposition}

\begin{proof}
	If $C_1\cap C_2\neq \emptyset$, then there exists $w^*\in C_1$ so that by Proposition~\ref{prop:P_C1}, $w^*=w^*-T^{\dagger}(Tw^*-\sqrt{2}c)$. Since $T$ has full column rank, then $T^{\dagger}T=I_{2n}$. Thus, $w^* = \sqrt{2}T^{\dagger}c$ is the unique point in $C_1$ and $\sqrt{2}c = Tw^* = \sqrt{2}TT^{\dagger}c$. Moreover, since $C_1\cap C_2$ is nonempty, then $w^*$ must be in $C_2$, i.e $\sqrt{2}T^{\dagger}c \in C_2$. Conversely, $TT^{\dagger}c=c$ and $\sqrt{2}T^{\dagger}c\in C_2$ implies that $\sqrt{2}T^{\dagger}c\in C_1\cap C_2$. The convergence of any sequence generated by \eqref{eqn:MAP_2} is an immediate consequence of Proposition~\ref{prop:P_C1}. In particular, $w^k=\sqrt{2}T^{\dagger}c$ for all $k\geq 1$ given any initial point $w^0 \in \Re^n\times \Re^n$. 
\end{proof}

\medskip 
Another specific case when we obtain global convergence can be obtained when $0\in C_1\cap C_2$. 

\begin{proposition}\label{prop:global_c=0}
	If $c=0\in \Re^m$, then any sequence generated by \eqref{eqn:MAP_2} converges to a point in $\Fix (P_{C_1}\circ P_{C_2})$. 
\end{proposition}

\begin{proof}
	This is a direct consequence of~\cite[Corollary~4.3]{DT19}. 
\end{proof}

A result stronger than the above proposition is derived in the next section (see Remark~\ref{remark:implicationofC-fcntheorem}). In particular, we shall see that any sequence of MAP iterates generated by \eqref{eqn:MAP_2} will always converge to a point in $C_1\cap C_2$ for any initial point $w^0$ whenever $c=0$. That is, MAP is globally convergent to a solution of the feasibility problem \eqref{eqn:AVE_as_a_FP2} for homogeneous AVEs.

\subsection{Convergence analysis using a new $C$-function}\label{subsec:newCfunction}

We now provide an alternative convergence analysis for the method of alternating projections by introducing a $C$-function that is new to the literature. We recall first the notion of $C$-functions. 

\medskip 
\begin{definition}\label{defn:Cfunction}
	A function $\phi:\Re^2 \to \Re$ is called a \emph{complementarity function} (or a \emph{$C$-function}) if its zeros are precisely the points on the nonnegative axes, i.e.
	\[\phi(s,t) = 0 \quad \Longleftrightarrow \quad s\geq 0, ~t\geq 0,~\text{and}~st=0.\]
\end{definition}
\medskip 

There are several examples of $C$-functions~\cite{AC20,Gal12}, as well as methods to construct these functions~\cite{ALNCC20}. Popular choices include the natural residual (NR) function and the Fischer-Burmeister (FB) function given respectively by 
\[\phi_{_{\rm NR}}(s,t) = \min (s,t) \qquad \text{and} \qquad \phi_{_{\rm FB}}(s,t) = \sqrt{s^2+t^2} - (s+t). \]
Given any $C$-function $\phi$, we define $\Phi: \Re^{n}\times \Re^n \to \Re^{n}$ as 
\[\Phi (u,v):= \left( \begin{array}{c}
\phi (u_1,v_1)  \\ 
\vdots \\
\phi (u_n,v_n) 
\end{array}\right).\]
It is then easy to see that 
	\begin{eqnarray}
	(u^*,v^*) \in C_2 \quad & \Longleftrightarrow & \quad \Phi(u^*,v^*) = 0 \notag \\
	& \Longleftrightarrow & \quad (u^*,v^*)\in \argmin _{(u,v)\in \Re^n \times \Re^n} \Psi (w) := \frac{1}{2}\|\Phi(u,v)\|^2. \notag 
	\end{eqnarray}
Consequently, the feasibility problem \eqref{eqn:AVE_as_a_FP2} can be equivalently reformulated as a constrained minimization problem
\begin{equation}\label{eqn:meritfunction_AVE}
\min _{w\in C_1} \Psi (w),
\end{equation}
provided that $C_1\cap C_2\neq \emptyset$. Note that if we define $\psi : = \frac{1}{2}\phi^2$, then $\psi$ is also a $C$-function and $\Psi (w) = \sum_{i=1}^n \psi (u_i,v_i)$. 

\medskip 
Although different $C$-functions yield different formulations \eqref{eqn:meritfunction_AVE}, a suitable choice of $\phi$ (or $\psi$) can facilitate the convergence analysis of MAP. Inspired by the equivalence of the method of alternating projections and the projected gradient method in the case of sparse affine feasibility problem as discussed in~\cite{HLN14}, we aim to choose a suitable $C$-function $\psi$ such that the induced function $\Psi$ satisfies 
\[(P_{C_1}\circ P_{C_2})(w) = P_{C_1}\left(w-\nabla \Psi (w) \right) ,\]
(where $\Psi$ should be differentiable to begin with). Unfortunately, $P_{C_2}$ is multivalued as shown in Proposition~\ref{prop:P_C2} while the right-hand side of the above equation is single-valued. Thus, we instead find a $C$-function which induces a function $\Psi$ satisfying 
\begin{equation}\label{eqn:desired_inclusion}
(P_{C_1}\circ P_{C_2})(w) \subseteq  P_{C_1}\left(w-\partial  \Psi (w)\right) , 
\end{equation}
where $\partial \Psi (w)$ denotes the Clarke generalized gradient of $\Psi:\Re^n \times \Re^n \to \Re$ at $w$. 

\medskip 
\begin{definition}\label{defn:clarke}~\cite{FP03}
	Let $F:\Re^n \to \Re$ be locally Lipschitz continuous on $\Re^n$. 
	\begin{description}
		\item[(a)] The \emph{$B$-subdifferential} of $F$ at $x$, denoted by $\partial_BF(x)$ is given by 
		\[\partial_B F(x) := \left\lbrace \lim _{x^k\to x} \nabla F(x^k)~:~ F~\text{is differentiable at } x^k \in \Re^n \right\rbrace.\]
		\item[(b)] The \emph{Clarke generalized gradient} of $F$ at a point $x\in \Re^n$, denoted by $\partial F(x)$, is defined as the convex hull of $\partial_B F(x)$.
	\end{description}
\end{definition}
\medskip 

In the next example, we illustrate that the NR and FB functions do not satisfy condition \eqref{eqn:desired_inclusion}.

\medskip 
\begin{example}
	Let $n=1$, $A=1$, $B=0$, $c=0$ and consider the point $w=(-1,-1)$. Denote the function $\Psi$ induced by the NR and FB functions by $\Psi_{_{\rm NR}}$ and $\Psi_{_{\rm FB}}$, respectively. From Definition~\ref{defn:clarke}, one can verify that 
	\[\partial_B \Psi_{_{\rm NR}}(w) = \{ (-1,0), (0,-1) \} \quad \text{and} \quad \partial_B \Psi_{_{\rm FB}} (w) = \{ (-3-2\sqrt{2}, -3-2\sqrt{2})\}.\]
	Thus, 
	\[P_{C_1}(w-\partial \Psi_{_{\rm NR}}(w)) = \{(-0.5,-0.5) \} \quad \text{and} \quad P_{C_1}(w-\partial \Psi_{_{\rm FB}}(w)) = \{(2+2\sqrt{2}, 2+2\sqrt{2}) \} .\]
	Meanwhile, we have $(P_{C_1}\circ P_{C_2})(w) = \{ (0,0)\}$. 
\end{example}
\medskip 

In the following result, we propose a $C$-function that is new to the literature and gives the desired inclusion \eqref{eqn:desired_inclusion}. 

\begin{proposition}\label{prop:new_C-function}
	The function defined by 
	\begin{equation}\label{eqn:new_C-function}
	\psi (s,t) = \begin{cases}
	\ds \frac{s^2}{2} + \frac{(-t)_+^2}{2} & \text{if}~s\leq t, \\ 
	\ds \frac{t^2}{2} + \frac{(-s)_+^2}{2} & \text{if}~s>t
	\end{cases} = \frac{\min (s,t)^2}{2} + \frac{\max (-\max(s,t),0)^2}{2}	
	\end{equation}
	is a nonnegative $C$-function. Moreover, $\psi$ is differentiable on $K_1\cup K_2$, where $K_1$ and $K_2$ are given by \eqref{eqn:K1_a>b} and \eqref{eqn:K2_a<b}, respectively, and the $B$-subdifferential of $\psi$ is given by
	\begin{equation}\label{eqn:B-sub_of_psi}
	\partial _B \psi (s,t) = \begin{cases}
	\{ (s,-(-t)_+) \} & \text{if}~s<t ~\text{or}~s=t\leq 0 \\
	\{ (-(-s)_+, t) \} & \text{if}~s>t \\ 
	\left\lbrace (s,0),(0,t) \right\rbrace & \text{if}~s = t >0 .
	\end{cases}
	\end{equation}
\end{proposition}

\begin{proof}
	Due to the symmetry of $\psi$ (that is, $\psi (s,t) = \psi (t,s)$), we only need to verify the equivalence in Definition~\ref{defn:Cfunction} for $s\leq t$. In this case, 
	\[\psi (s,t) = 0 \quad \Longleftrightarrow \quad s=0 ~ \text{and}~(-t)_+= 0 \quad \Longleftrightarrow \quad s=0~\text{and}~t\geq 0.\]
	This proves that $\psi$ is a $C$-function. It can also be verified that $\psi$ is locally Lipschitz continuous on $\Re^2$ (see also~\cite[Lemma~4.6.1]{FP03} or~\cite[Proposition~4.1.2]{Scholtes12}). Next, note that $\psi$ is differentiable only on $K_1\cup K_2$. The first two cases in formula \eqref{eqn:B-sub_of_psi} can be easily verified. If $s=t>0$ and $\{(s^k,t^k) \}_{k=1}^{\infty}$ is a sequence in $K_1\cup K_2$ converging to $(s,t)$, then for sufficiently large $k$, the sequence lie in $\Re^2_{++}$. Hence, the only subsequential limits of $\{\nabla \psi (s^k,t^k) \}_{k=1}^{\infty}$ are the limits of $\{(s^k, 0 ) \}_{k=1}^{\infty}$ and $\{(0, t^k ) \}_{k=1}^{\infty}$, which are $(s,0)$ and $(0,t)$, respectively. This completes the proof.  
\end{proof}

\medskip 
We next show that the induced function $\Psi (w) $ of \eqref{eqn:new_C-function} indeed gives the desired inclusion \eqref{eqn:desired_inclusion}. In fact, the following corollary shows that the MAP iterates \eqref{eqn:MAP_2} are the same as the ``projected $B$-subdifferential'' iterates.

\medskip 
\begin{corollary}\label{cor:MAP_equiv_PGM}
	If $\psi$ is given by \eqref{eqn:new_C-function} and $\Psi:\Re^{n}\times \Re^n \to \Re_+$ is given by $\Psi (w) := \ds \sum _{i=1}^n \psi (u_i,v_i)$, then 
	\[P_{C_2}(w) =  w-\partial_B  \Psi (w).   \]
	In particular, \eqref{eqn:desired_inclusion} holds. 
\end{corollary}

\begin{proof}
	Denote $w=(u,v)\in \Re^n \times \Re^n$. A direct verification shows that $\partial \Psi_B (w) = \sum _{i=1}^n D^i,$ where the summation denotes the Minkowski sum of sets, and $D^i\subseteq \Re^{n}\times \Re^n$ denotes the set of all $d^i$ such that 
	\[(d^i_j , d^i_{n+j}) \in \begin{cases}
	\{ (0,0)\} & \text{if}~j\neq i \\
	\partial _B \psi (u_i,v_i) & \text{if}~j=i
	\end{cases}, \qquad i=1,2,\dots, n  \]
	with $\partial _B \psi (u_i,v_i)$ given by \eqref{eqn:B-sub_of_psi}. To establish the result, we only need to show that $P_{M}(s,t) = (s,t) - \partial_B \psi (s,t)$, where $P_{M}$ is given by \eqref{eqn:P_C2bar}. This equality can be directly verified by using the fact that $x + (-x)_+ = x_+$ for all $x\in \Re$.
\end{proof}

We next establish one more important property of $\psi$ as defined in \eqref{eqn:new_C-function} which will later be useful in proving our convergence result. 

\begin{lemma}\label{lemma:keyinequality}
	Let $\psi$ be given by \eqref{eqn:new_C-function} and let $(a,b), (s,t) \in \Re^2$. If $(a,b), (s,t) \in K_1$ or $(a,b), (s,t) \in K_2$, where $K_1$ and $K_2$ are given by \eqref{eqn:K1_a>b} and \eqref{eqn:K2_a<b}, respectively, then 
	\begin{multline}\label{eqn:psi_inequality1}
	\psi (s,t) - \psi (a,b)  \geq  \frac{1}{2} \lla \nabla \psi (a,b) , (s-a,t-b) \rla \\ - \frac{\min (a,b)^2}{8}  - \frac{\max (a,b)^2}{8} \mathbbm{1}_{\Re^2_{-}}(\max(a,b),\max(s,t)),
	\end{multline}
	where $\mathbbm{1}_{\Re^{2}_{-}}(c,d)=1$ if $(c,d)\in \Re^2_{-}$ and $0$ otherwise. In particular, if $\psi(s,t)=0$, then 
	\begin{equation}\label{eqn:psi_inequality2}
	2\psi (a,b) \leq \lla \nabla \psi (a,b) , (a-s,b-t) \rla . 
	\end{equation}
	Moreover, 
	\begin{equation}\label{eqn:psi_inequality3}
	2\psi (a,b) \leq \lla  \psi' (a,b), (a,b)\rla \qquad \forall (a,b)\in \Re^2, ~\forall \psi'(a,b)\in \partial_B \psi (a,b).
	\end{equation}
\end{lemma}

\begin{proof}
	By symmetry of $\psi$, it suffices to consider the case when $(a,b),(s,t)\in K_1$ to prove \eqref{eqn:psi_inequality1}. By direct computation, we get from \eqref{eqn:new_C-function} and \eqref{eqn:B-sub_of_psi} that 
	\begin{equation*}
	\psi (s,t) - \psi (a,b) - \frac{1}{2} \lla \nabla \psi (a,b) , (s-a,t-b) \rla  = \frac{t^2}{2} - \frac{bt}{2} +  \frac{(-s)_+^2}{2} + \frac{(-a)_+ s}{2} 
	\end{equation*}
	Noting that $t^2-bt\geq -b^2/4$ and $s^2-as \geq -a^2/4$, we get the desired inequality. On the other hand, \eqref{eqn:psi_inequality2} directly follows from \eqref{eqn:psi_inequality2}. Finally, in view of \eqref{eqn:psi_inequality2}, we only need to verify inequality \eqref{eqn:psi_inequality3} for $a=b>0$ which is a routine calculation.  
\end{proof}

\medskip 
We now present our convergence result using the $C$-function \eqref{eqn:new_C-function}. 

\begin{theorem}\label{theorem:local2}
	Let $\{w^k\}_{k=0}^{\infty}$ be any sequence generated by \eqref{eqn:MAP_2}. Suppose $w^*=(u^*,v^*)\in C_1\cap C_2$, and denote 
	\begin{eqnarray}
	I_1^* & := & \{i:u_i^*>v_i^*=0\} , \notag \\
	I_2^* & := & \{i:0=u_i^*<v_i^*\} , \notag 
	\end{eqnarray}
	and let $\Gamma^* := \{ w=(u,v): (u_i,v_i)\in K_i~if~i\in I_i^*~(i=1,2)\}$. If $\ds w^k \in \Gamma^* $ for all sufficiently large $k$, then $\Psi(w^k) \to 0$ as $k\to\infty$. Moreover, there exists a point $\bar{w}\in C_1\cap C_2$ such that $w^k\to \bar{w}$ as $k\to\infty$.  
\end{theorem}

\begin{proof}
	We have from Corollary~\ref{cor:MAP_equiv_PGM}  that $w^k - \Psi '(w^k) \in P_{C_2}(w^k)$ where $\Psi '(w^k) \in \partial_B \Psi (w^k)$. Thus, 
	\begin{eqnarray}
	\|w^{k+1}-w^*\|^2 &  =  & \| P_{C_1}(w^k - \Psi' (w^k)) - P_{C_1}(w^*)\|^2 \notag \\
	& \leq & \|(w^k-w^*)-\Psi' (w^k)\|^2 \notag \\ 
	& = & \|w^k-w^*\|^2 - 2 \lla w^k-w^*, \Psi' (w^k) \rla + \| \Psi' (w^k)\|^2 , \label{eqn:upperbound_w^k-w^*}
	\end{eqnarray}
	where the inequality holds by nonexpansiveness of $P_{C_1}$. Meanwhile, since $w^k, w^*\in \Gamma^*$, then inequalities \eqref{eqn:psi_inequality2} and \eqref{eqn:psi_inequality3} yield
	\begin{eqnarray}
	\lla w^k-w^*, \Psi' (w^k) \rla & = & \lla (u^k-u^*,v^k-v^*), \Psi' (w^k) \rla \notag \\
	& = & \sum _{i\in I_1^* \cup I_2^*} \lla (u_i^k - u_i^*,v_i^k-v_i^*),\nabla \psi (u_i^k,v_i^k) \rla \notag \\
	& & \qquad +  \sum _{i\notin I_1^* \cup I_2^*} \lla (u_i^k,v_i^k), \psi '(u_i^k,v_i^k) \rla , \quad \psi '(u_i^k,v_i^k) \in \partial_B \psi (u_i^k, v_i^k) \notag \\ 
	& \geq & 2\sum_{i=1}^n \psi (u_i^k,v_i^k) \notag \\
	& = & 2\Psi (w^k). \label{eqn:upperbound_Psi(w^k)}
	\end{eqnarray}
	On the other hand, we have
	\begin{equation}\label{eqn:2Psi-DPsi^2}
	\| \Psi '(w^k)\|^2 - 2\Psi (w^k) = \sum _{i=1}^n \left[ \|  \psi' (u_i^k,v_i^k)\|^2 - 2\psi (u_i^k,v_i^k)\right] =0,
	\end{equation}
	where the last equality can be verified directly from \eqref{eqn:new_C-function} and \eqref{eqn:B-sub_of_psi}. Continuing from \eqref{eqn:upperbound_w^k-w^*}, we have 
	\begin{eqnarray}
	\|w^{k+1}-w^*\|^2 & \leq & \|w^k-w^*\|^2 - 4\Psi (w^k) + \|\Psi' (w^k)\|^2 \label{eqn:upperbound_w^k-w^*2}  \\
	& = & \|w^k - w^*\|^2 - 2\Psi (w^k) \label{eqn:finalinequality}
	\end{eqnarray}
	where \eqref{eqn:upperbound_w^k-w^*2} and \eqref{eqn:finalinequality} follow from \eqref{eqn:upperbound_Psi(w^k)} and \eqref{eqn:2Psi-DPsi^2}, respectively. From \eqref{eqn:finalinequality}, we get 
	\begin{eqnarray*}
		2 \sum _{k=1}^N \Psi (w^k) & = & \|w^0 - w^*\|^2 - \|w^{N+1} - w^*\|^2 \leq \| w^0 - w^* \|^2  \qquad \forall N\in \mathbb{N} .
	\end{eqnarray*} 
	Thus, $\Psi (w^k)\to 0$ as $k\to \infty$. This proves the first claim. 
	
	Meanwhile, note that \eqref{eqn:finalinequality} implies that the sequence $\{w^k\}_{k=0}^{\infty}$ is bounded. Thus, the sequence has an accumulation point $\bar{w}$, i.e. there exists a subsequence $\{w^{k_j}\}_{j=1}^{\infty}$ such that $w^{k_j}\to \bar{w}$ as $j\to\infty$. Since $\{w^k\}_{k=0}^{\infty}\subseteq C_1$ and $C_1$ is closed, then $\bar{w}\in C_1$. Moreover, $\Psi(w^{k_j}) \to \Psi (\bar{w})$ as $j\to \infty$ since $\Psi$ is continuous. Since the full sequence $\{\Psi (w^k)\}_{k=0}^{\infty}$ converges to zero, then $\Psi(\bar{w})=0$, that is, $\bar{w}\in C_2$. Hence, we obtain that $\bar{w}\in C_1\cap C_2$ and since $\bar{w}$ must also be a point in the closure of  $\Gamma^*$, then $\bar{w}\in \Gamma^*$. By applying the same argument for $w^*$ as above, we obtain $\|w^{k+1}-\bar{w} \| \leq \|w^k-\bar{w}\|$ as in \eqref{eqn:finalinequality}. Thus, $\{\|w^k-\bar{w}\| \}$ is a decreasing sequence of nonnegative numbers and must therefore be convergent. Since $\|w^{k_j}-\bar{w}\| \to 0$ as $j\to \infty$, then it follows that $\|w^k-\bar{w}\|\to 0$, i.e. $w^k \to \bar{w}$ as $k\to \infty$. This completes the proof. 
\end{proof}

\begin{remark}\label{remark:implicationofC-fcntheorem}
	We observe that the above theorem implies the local convergence result given by Theorem~\ref{theorem:local_general}. In addition, we obtain the global convergence to $C_1\cap C_2$ if $c=0$, which is stronger than the claim of Proposition~\ref{prop:global_c=0}. Hence, the above discussion provides an alternative proof for the aforementioned results in light of the new $C$-function $\psi$.
	
	\medskip 
	
	To see precisely how one gets the local convergence given by Theorem~\ref{theorem:local_general}, let $B(w,\delta)$ denote the open ball centered at $w$ with radius $\delta$. For each $i\in I_1^*$, let $\delta_i>0$ be such that $B((u_i^*,v_i^*),\delta_i) \subseteq \{ (a,b)\in \Re^2 :a>b\}$ and for each $i\in I_2^*$, let $\delta_i>0$ so that $ B((u_i^*,v_i^*),\delta_i) \subseteq \{ (a,b)\in \Re^2:a<b\}$. Taking $\delta := \min \{ \delta_i:i\in I_1^*\cup I_2^*\}$, then $B (w^*,\delta)\subseteq \Gamma^*$. Moreover, for any $w \in \Gamma^*$, note that 
	\begin{eqnarray}
	\|P_{C_2}(w) - P_{C_2}(w^*) \|^2 & = & \sum_{i\in I_1^*\cup I_2^*} \|P_{M}(u_i,v_i) -P_{M}(u_i^*,v_i^*) \|^2 \notag \\
	& & \qquad  + \sum_{i\notin I_1^*\cup I_2^*} \|P_{M}(u_i,v_i) -P_{M}(0,0) \|^2 \notag \\
	& \leq  & \sum_{i\in I_1^*\cup I_2^*} \|(u_i,v_i) -(u_i^*,v_i^*) \|^2 + \sum_{i\notin I_1^*\cup I_2^*} \|P_{M}(u_i,v_i) \|^2 \notag \\
	& \leq  & \sum_{i\in I_1^*\cup I_2^*} \|(u_i,v_i) -(u_i^*,v_i^*) \|^2 + \sum_{i\notin I_1^*\cup I_2^*} \|(u_i,v_i) \|^2 \notag \\	
	& = & \|w-w^*\|^2, \label{eqn:nonexpansiveatw^*,overGamma*}
	\end{eqnarray}
	where $M$ is as defined in the proof of Proposition~\ref{prop:P_C2}. The first inequality above follows from the proof of Corollary~\ref{cor:PC2_nonexpansive}, while the second inequality holds since $\|P_{M}(a,b)\| \leq \|(a,b)\|$ for all $(a,b)\in \Re^2$. It follows from inequality \eqref{eqn:nonexpansiveatw^*,overGamma*} that if $w^0\in B(w^*,\delta)$, then $w^k\in B(w^*,\delta)$ for all $k\geq 0$. Thus, $w^k\in \Gamma^*$ for all $k$ and by Theorem~\ref{theorem:local2}, $w^k$ converges to some $\bar{w}\in C_1\cap C_2$. This is precisely the claim of Theorem~\ref{theorem:local_general}. 
	
	\medskip 
	
	If $c=0$, we see that $w^*=0 \in C_1\cap C_2$ so that $I_1^*=I_2^*=\emptyset$. The above discussion reveals that the MAP iterates will converge to a point in $C_1\cap C_2$ given any initial point $w^0$. $_{\blacksquare }$ 
\end{remark}

\medskip 
We provide a geometric interpretation of Theorem~\ref{theorem:local2} for the case that $I_1^*\cup I_2^* = \{1,\dots,n\}$. In this case, there exists $\tau^*\in \mathscr{T}$ such that $w^*$ is contained in the interior of $S_{\tau^*}$. Theorem~\ref{theorem:local2} indicates that if the iterates will eventually be ``trapped'' in $S_{\tau^*}$, then the iterates must converge to a solution of the feasibility problem \eqref{eqn:AVE_as_a_FP2}. Observe that the above remark implies that this could occur if we choose an initial point $w^0$ that is close enough to $w^*$. However, it is in general difficult to prove this when the initial point is arbitrarily set. Nevertheless, we will prove that for the case $m=n$, it is impossible for the iterates to be eventually trapped in some $S_{\tau}$ that does not contain a point in $C_1\cap C_2$ if we assume nondegeneracy of $Q$ as defined in \eqref{eqn:Qmatrix}. To this end, we need the following lemma.

\medskip 

Given any matrix $A \in \Re ^{m\times n}$, we denote by $\sigma_k(A)$ the $k$th largest singular value of $A$. Moreover, the norm of $A$ is the largest singular value, i.e. $\|A\| =\sigma_1(A)$. If $k>\min \{m,n\}$, we set $\sigma _k (A) = 0$. 

\begin{lemma}\cite[Corollary~3.1.3]{HJ91}\label{lemma:singularvalue_submatrix}
	Let $A\in \Re^{m\times n}$ and let $A_r$ denote a submatrix of $A$ obtained by deleting a total of $r$ rows and/or columns of $A$. Then 
	\[\sigma_k (A)\geq \sigma_k (A_r) \geq \sigma _{k+r}(A), \quad k=1,\dots, \min\{m,n\}.\]
\end{lemma}

\medskip 
\begin{lemma}\label{lemma:property_Q2}
	Suppose that $Q$ given by \eqref{eqn:Qmatrix} is nondegenerate. Let $\Lambda := \Lambda_1 \cup \Lambda_2$ where $\Lambda_1\subseteq \{1,\dots, n\}$ and $\Lambda_2=\{ n+i : i \notin \Lambda_1 \}$. If $L_{\cdot \Lambda}$ is the submatrix of $L:=I_{2n}-T^{\dagger}T$ containing all its columns indexed by $\Lambda$ and all of its $2n$ rows, then $\|L_{\cdot \Lambda}\|<1$. 
\end{lemma}

\begin{proof}
	Let $E_1\in \Re^{n\times n}$ such that the first $|\Lambda_1|$ columns of $E_1$ are the standard unit vectors $e_i\in \Re^n$ with $i\in \Lambda_1$, while the other remaining columns are zeros. In addition, let $E_2\in \Re^{n\times n}$ be such that the first $|\Lambda_1|$ columns are zeros and the last $|\Lambda_2|$ columns are composed of $e_i$'s where $i\notin \Lambda_1$. Further, let $E:= \begin{bmatrix}
	E_1 & E_2 \\
	E_2 & E_1
	\end{bmatrix}
	$. We note that 
	\begin{equation}\label{eqn:identities_Ematrices}
	E_1E_1^\T + E_2E_2^\T = I_n, \qquad E_i E_j^\T=0 ~(\forall i\neq j) \qquad \text{and} \qquad EE^\T = E^\T E = I_{2n}.
	\end{equation}
	Then the matrix $L_{\cdot \Lambda}$ is precisely the submatrix of $\tilde{L}:= E^\T ME$ containing all its rows and its first $n$ columns. Meanwhile, using the identities \eqref{eqn:identities_Ematrices}, it can be verified that the matrix $\tilde{L}$ is also equal to $ I_{2n} - \tilde{T}^{\dagger}\tilde{T}$ where $\tilde{T}:=TE = [~U~~V~]$,  $U:=T \left[ \begin{array}{c}
	E_1 \\ E_2 
	\end{array} \right]$ and $V:=T \left[ \begin{array}{c}
	E_2 \\ E_1 
	\end{array} \right]$. Calculating $\tilde{L}$ using this formula, we see that 
	\[L_{\cdot \Lambda} = \left[ \begin{array}{c}
	I_n - U^\T W U \\ -V^\T W U
	\end{array} \right],\]
	where $W=(TT^\T )^{-1}$. Noting that $UU^\T + VV^\T = W^{-1}$, which can be derived from \eqref{eqn:identities_Ematrices}, we obtain 
	\begin{equation}\label{eqn:identity_UWU}
	U^\T W V V^\T W U = U^\T W U - (U^\T W U )^2.
	\end{equation}
	Then
	\begin{eqnarray*}
		\|L_{\cdot \Lambda}\|^2 & = & \lambda _{\max} (L_{\cdot \Lambda}^\T L_{\cdot \Lambda}) \\
		& = & \lambda _{\max} \left(  (I_n - U^\T W U )^2 + U^\T W V V^\T W U  \right) \\
		& = & \lambda_{\max} (I_n - U^\T W U)
	\end{eqnarray*}
	where the last equality follows from \eqref{eqn:identity_UWU}. Meanwhile, we know from the definition of $U$ that it is composed of the columns of $T$ which are indexed by $\Lambda$. By Lemma~\ref{lemma:Pnondegenerate_relationwithD}, these columns must be linearly independent so that $U$ is nonsingular and $1$ is not an eigenvalue of $I_n - U^\T W U$. We conclude that $\|L_{\cdot \Lambda}\| \neq 1$. But by Lemma~\ref{lemma:singularvalue_submatrix}, we know that $\|L_{\cdot \Lambda}\|\leq \|L\| =  1$. Hence, we arrive at the desired conclusion. 
\end{proof}

\medskip

Using the above lemma, we obtain the following proposition.

\begin{proposition}\label{prop:impossible_to_be_trapped_in_wrong_region}
	Let $\{w^k\}_{k=0}^{\infty}$ be any sequence generated by \eqref{eqn:MAP_2}, and suppose there exists $\tau \in \mathscr{T}$ such that $S_{\tau}\cap (C_1\cap C_2)=\emptyset$, i.e. $S_{\tau}$ does not contain a solution of \eqref{eqn:AVE_as_a_FP2}. Then there does not exist $N\in \mathbb{N}$ such that $\{w^k\}_{k=N}^{\infty} \subseteq S_{\tau}\cap \Omega$, where $\Omega$ is as defined in Theorem~\ref{theorem:fixedpoints_arbitrary}. 
\end{proposition}

\begin{proof}
	Suppose to the contrary that there exists $N$ such that $w^k \in S_{\tau} \cap \Omega$ for all $k\geq N$. To prove the result, we will show that $w^k$ converges to some point $w^*\in S_{\tau}\cap (C_1 \cap C_2)$ which is a contradiction to our hypothesis. To this end, we apply a convenient change of variables based on $\tau$. First, let $\Lambda_1 := \{ i:\tau (i)=1\}$ and $\Lambda_2:= \{ n+i : \tau (i)=2\}$. With these index sets, define the matrices $E$, $L_{\cdot \Lambda}$, $\tilde{L}$ and $\tilde{T}$ be as in the proof of Lemma~\ref{lemma:property_Q2}. 
	
	We consider the transformation $w = E\tilde{w}$. Similar to the discussion in Section \ref{subsec:changeofvariables}, we see that $w\in C_i$ if and only if $\tilde{w}\in \tilde{C}_i$ where $\tilde{C}_1 := \{ \tilde{w}: \tilde{T}\tilde{w} = \sqrt{2}c\}$ and $\tilde{C}_2=C_2$. Moreover, 
	\begin{equation}\label{eqn:Stau_switched}
	w\in S_{\tau }\cap \Omega \qquad \Longleftrightarrow \qquad \tilde{u}_i > \tilde{v}_i ~\text{and}~\tilde{u}_i\geq 0 \quad \forall i=1,\dots, n.
	\end{equation}
	Analogous to equations \eqref{eqn:StoC} and \eqref{eqn:fixedpoints_transformed}, we have 
	\begin{equation*}\label{eqn:PC1PC2_to_tilde}
	(P_{C_1}\circ P_{C_2})(w)  = E(P_{\tilde{C}_1}\circ P_{\tilde{C}_2} ) (\tilde{w}),
	\end{equation*}
	and
	\begin{equation}\label{eqn:fixedpoints_transformed2}
	\Fix (P_{C_1}\circ P_{C_2}) = E \left( \Fix (P_{\tilde{C}_1}\circ P_{\tilde{C}_2}) \right) 
	\end{equation}
	since $E$ is unitary. 
	\medskip 
	We now look at the transformed iterates $\tilde{w}^k = E^{\T}w^k = (\tilde{u}_k, \tilde{v}_k)$. By \eqref{eqn:Stau_switched}, we have $\tilde{u}_k > \tilde{v}_k$ and $\tilde{u}_k\geq 0$ for all $k\geq N$ so that 
	\begin{eqnarray}
	\tilde{w}^{k+1} & = & (P_{\tilde{C}_1}\circ P_{\tilde{C}_2} )(\tilde{w}^k) \label{eqn:iterates_switchedcoordinates}\\
	& = & P_{\tilde{C}_1}\left( (\tilde{u}^k, 0)\right) \label{eqn:PC2tilde}\\
	& = & \tilde{L} (\tilde{u}^k,0)^{\T} + \sqrt{2}\tilde{T}^{\dagger}c \label{eqn:PC1tilde}\\
	& = & L_{\cdot \Lambda} \tilde{u}^k + \sqrt{2}\tilde{T}^{\dagger}c \qquad \forall k\geq N, \notag 
	\end{eqnarray}
	where \eqref{eqn:PC2tilde} and \eqref{eqn:PC1tilde} follow from Propositions \ref{prop:P_C2} and \ref{prop:P_C1}, respectively. Letting $L_{\cdot \Lambda} = \begin{bmatrix} 
	(L_{\cdot \Lambda})_1 \\ (L_{\cdot \Lambda})_2 
	\end{bmatrix}$ and $\sqrt{2}\tilde{T}^{\dagger}c = \begin{bmatrix} 
	d_1 \\ d_2 
	\end{bmatrix}$ where $(L_{\cdot \Lambda})_1,(L_{\cdot \Lambda})_2\in \Re^{n\times n}$ and $d_1,d_2 \in \Re^n$, then we see that $\tilde{u}^{k+1} = (L_{\cdot \Lambda})_1 \tilde{u}^k + d_1$ and $\tilde{v}^{k+1} = (L_{\cdot \Lambda})_2 \tilde{u}^k + d_2$. Since $\|L_{\cdot \Lambda}\|<1$ by Lemma~\ref{lemma:property_Q2}, we also have by Lemma~\ref{lemma:singularvalue_submatrix} that $\|(L_{\cdot \Lambda})_1\|<1$ so that $\{\tilde{u}^k\}_{k=0}^{\infty}$ is convergent. Consequently, $\{\tilde{v}^k\}_{k=0}^{\infty}$ is also convergent. 
	
	Therefore, there exists $\tilde{w}^*$ such that $\tilde{w}^k\to \tilde{w}^*$ as $k\to \infty$. Moreover, we have from \eqref{eqn:iterates_switchedcoordinates} that $\tilde{w}^* = (P_{\tilde{C}_1}\circ P_{\tilde{C}_2} )(\tilde{w}^*)$, i.e. $\tilde{w}^* \in \Fix (P_{\tilde{C}_1}\circ P_{\tilde{C}_2})$. Since $w^k = E\tilde{w}^k$, it also follows that $w^k \to w^* := E\tilde{w}^*$ and from \eqref{eqn:fixedpoints_transformed2}, $w^* \in \Fix (P_{C_1}\circ P_{C_2})$. Since $\Omega$ is closed, it also follows that $w^* \in \Omega$. From Theorem~\ref{theorem:fixedpoints}, we must have $w^* \in C_1\cap C_2$. However, since $w^*$ must belong to the closure of $S_{\tau}$, the fact that it is in $C_1\cap C_2$ implies that $w^* \in S_{\tau}$. Hence, $w^* \in S_{\tau} \cap (C_1\cap C_2)$. This is a contradiction. 
\end{proof}

\subsection{Rate of Convergence}\label{subsec:rateofconvergence}

An immediate consequence of Lemma~\ref{lemma:property_Q2} is the linear rate of convergence of MAP iterates \eqref{eqn:MAP_2}. 
\medskip 

\begin{theorem}\label{theorem:linearrate}  
	Let $m=n$ and suppose that $w^* \in C_1 \cap C_2$ such that $(u_i^*,v_i^*) \neq (0,0)$ for all $i=1,\dots, n$. If $Q$ given by \eqref{eqn:Qmatrix} is nondegenerate, then there exists sufficiently small $\delta >0$ such that for any $w^0$ with $\|w^0-w^*\|<\delta$, the sequence $\{w^k\}_{k=0}^{\infty}$ generated by \eqref{eqn:MAP_2} converges linearly to $w^*$. 
\end{theorem}

In the following, we denote by $\supp (w) := \{ i:w_i\neq 0\}$ the support of a vector $w$. 

\begin{proof}
	Let $\tau^* \in \mathscr{T}$ such that $w^* \in S_{\tau^*}$. Observe that since $(u_i^*,v_i^*) \neq (0,0)$ for all $i=1,\dots, n$, then $\Gamma^*$ defined in Theorem~\ref{theorem:local2} is precisely the set $S_{\tau^*}$. Choose $\delta>0$ sufficiently small so that the closure of 
	$B(w^*,\delta) $ is contained in the interior of $\Gamma^*=S_{\tau^*}$. From the discussion in Remark~\ref{remark:implicationofC-fcntheorem}, we have $\{w^k\}_{k=0}^{\infty} \subseteq B(w^*,\delta)$ whenever $w^0 \in B(w^*,\delta)$. Moreover, there exists $\bar{w} \in C_1\cap C_2$ in the interior of $S_{\tau^*}$ such that $w^k\to \bar{w}$ as $k\to\infty$ for any $w^0 \in B(w^*,\delta)$. Thus, $(\bar{u}_i,\bar{v}_i)\neq (0,0)$ for all $i=1,\dots, n$. Meanwhile, in view of the equivalence of AVE and the LCP described in Remark~\ref{remark:AVE_equiv_LCP} together with~\cite[Theorem~3.6.3]{CPS92}, the nondegeneracy assumption on $Q$ implies that $w^*$ is an isolated solution of the feasibility problem \eqref{eqn:AVE_as_a_FP2}. Thus, by choosing a smaller $\delta$ (if necessary), we have that $\bar{w}=w^*$. That is, $w^k \to w^*$ for all $w^0 \in B(w^*,\delta)$. 
	
	Denote by $\Lambda$ the support of $w^*$. Since $\{w^k\}_{k=0}^{\infty}$ is contained in the interior of $S_{\tau^*}$ and $w^k\to w^*$, then we have by Proposition~\ref{prop:P_C2} that $P_{C_2}$ is single-valued at $w^k$ and $\supp (P_{C_2}(w^k)) = \Lambda$ for all $k\geq 0$. Thus,
	\[\| (P_{C_2}(w^k) - P_{C_2}(w^*))_{\Lambda} \| = \| P_{C_2}(w^k) - P_{C_2}(w^*) \| \leq \| w^k - w^*\|, \]
	where the inequality holds by nonexpansiveness of $P_{C_2}$ on $S_{\tau^*}$ (Corollary~\ref{cor:PC2_nonexpansive}). Then if $L_{\cdot \Lambda}$ denotes the submatrix of $L:=I_{2n}-T^{\dagger}T$ containing all of its $2n$ rows and all its columns indexed by $\Lambda$, we have from Proposition~\ref{prop:P_C1} that 
	\begin{eqnarray}
	\| w^{k+1} - w^*\| & = &  \| P_{C_1}(P_{C_2}(w^k)) - P_{C_1}(P_{C_2}(\bar{w}))\| \notag \\
	& =  & \| L(P_{C_2}(w^k)) - L(P_{C_2}(w^*))\| \notag \\
	& = & \| L_{\cdot \Lambda} (P_{C_2}(w^k) - P_{C_2}(w^*))_{\Lambda} \| \notag \\
	& \leq  & \| L_{\cdot \Lambda}\| \cdot \| (P_{C_2}(w^k) - P_{C_2}(w^*))_{\Lambda} \| \notag \\
	& \leq &  \| L_{\cdot \Lambda}\| \cdot \|w^k - w^*\|, \notag 
	\end{eqnarray}
	Since $\|L_{\cdot \Lambda}\| <1$ by Lemma~\ref{lemma:property_Q2}, the conclusion of this theorem follows. 
\end{proof}

\medskip
The rate of convergence asserted by the above result can also be obtained using~\cite[Theorem~5.16]{LLM09} and Proposition~\ref{lemma:Pnondegenerate_relationwithD}. In fact, it can be extended to the general case when $m$ is not necessarily equal to $n$ using the notions of ``super-regularity'' and ``linearly regular intersection''. We recall from~\cite{LLM09} that a closed set $C$ is \textit{super-regular} at $w^*$ if, for all $\varepsilon>0$, any two points $z_1, z_2$ sufficiently close to $w^*$ with $z_2\in C$, and any point $y\in P_C(z_1)$, satisfy $\lla z_1-y, z_2-y \rla \leq \varepsilon \| z_1-y \| \cdot \|z_2-y\|$. In particular, a convex set is super-regular at each of its points. We refer the reader to~\cite{DLT19} for more details on how super-regularity is related with other pre-existing notions.

\medskip 
To define the concept involving sets with linearly regular intersection, we first recall that the \textit{limiting normal cone} to a closed set $C$ at $w^*\in C$ is given by 
\[N_{C}(w^*) = \left\lbrace \lim _{k\to\infty} t_k (w^k-z^k): t_k \geq 0, ~w^k \to w^*, ~z^k \in P_C(w^k) \right\rbrace . \]
We say that two closed sets $C_1$ and $C_2$ have a \textit{linearly regular intersection} at $w^* \in C_1\cap C_2$ if 
\begin{equation}\label{eqn:linearlyregularequation}
N_{C_1}(w^*) \cap \left( -N_{C_2}(w^*)\right) = \{ 0\}.
\end{equation}
With these definitions, we state the following convergence result from~\cite{LLM09}.

\begin{lemma}\cite[Theorem~5.16]{LLM09}\label{lemma:linearlyregular}
	If $C_1$ and $C_2$ are closed sets which have a linearly regular intersection at $w^* \in C_1\cap C_2$ and if either $C_1$ or $C_2$ is super-regular at $w^*$, then any alternating projection sequence with initial point sufficiently close to $w^*$ converges linearly to a point in $C_1\cap C_2$. 
\end{lemma}

\medskip 
Using the above result, we obtain the linear convergence of the MAP iterates.

\begin{theorem}
	Suppose that $w^* \in C_1\cap C_2$ such that $(u_i^*,v_i^*)\neq (0,0)$ for all $i=1,\dots, n$. If condition \eqref{eqn:linearlyregularintersection} holds, then any sequence generated by \eqref{eqn:MAP_2} with initial point sufficiently close to $w^*$ converges linearly to a point in $C_1\cap C_2$. 
\end{theorem}

\medskip 
\begin{proof}
	We note that since $C_1$ is convex, then it is super-regular at each of its points. Thus, by Lemma~\ref{lemma:linearlyregular}, it suffices to show that $C_1$ and $C_2$ have a linearly regular intersection at $w^*\in C_1\cap C_2$ where $(u_i^*,v_i^*)\neq (0,0)$ for all $i=1,\dots, n$. Directly from the definition, the limiting normal cones to $C_1$ and $C_2$ are given, respectively, by
	
	\[N_{C_1}(w) = \Ker (T)^{\perp} \qquad \forall w\in C_1\] and
	\begin{equation}\label{eqn:NormalconeC2}
	N_{C_2}(w) = \{ w'=(u',v')\in \Re^n \times \Re^n : (u_i',v_i')\in N_{M} (u_i,v_i)\} \qquad \forall w\in C_2
	\end{equation}
	where $M$ is given by \eqref{eqn:C2bar} and 
	\[N_{M}(s,t) = \begin{cases}
	\{ (0,\lambda): \lambda \in \Re \} & \text{if}~s>t=0 \\ 
	\{ (\lambda,0): \lambda \in \Re \} & \text{if}~0=s<t \\ 
	\Re^{2}_- \cup M & \text{if}~s=t=0 
	\end{cases} . \]
	We note that the normal cone to $C_2$ can also be obtained using~\cite[Theorem~3.4]{Tam17}. Since $(u_i^*,v_i^*)\neq (0,0)$, it follows that $N_{C_2}(w^*) \subseteq \hat{C}_2$, where $\hat{C}_2$ is given by \eqref{eqn:C2hat}. By condition \eqref{eqn:linearlyregularintersection}, we see that \eqref{eqn:linearlyregularequation} holds, i.e. the intersection at $w^*$ is linearly regular. This completes the proof. 
\end{proof}

\subsection{Globally convergent relaxation of MAP}\label{subsec:globalconvergence_relaxedMAP}
In the preceding sections, our analysis was focused on the MAP iterates given by \eqref{eqn:MAP_2}. In order to obtain a global result, we now focus on a relaxed version of the iterations \eqref{eqn:MAP_2} given by 
\begin{equation}\label{eqn:relaxedMAP}
w^{k+1} \in (1-\gamma)P_{C_2}(w^k) + \gamma (P_{C_1}\circ P_{C_2})(w^k)
\end{equation}
where $\gamma \in (0,1)$ is fixed and the initial point is $w^0 = (1-\gamma)\bar{w}^0 + \gamma P_{C_1}(\bar{w}^0)$ with $\bar{w}^0 \in C_2$.

\begin{theorem}\label{theorem:global}
	Let $\{w^k\}_{k=0}^{\infty}$ be a sequence generated by \eqref{eqn:relaxedMAP}. If $\{ w^k\}_{k=0}^{\infty}$ is bounded, then there exists $\bar{w}^*\in C_2$ such that $w^k \to w^*$ where $w^* = (1-\gamma)\bar{w}^* + \gamma P_{C_1}(\bar{w}^*)$. Moreover, if condition \eqref{eqn:linearlyregularintersection} holds and $(\bar{u}_i^*,\bar{v}_i^*) \neq (0,0)$ for all $i=1,\dots,n$, then $w^* \in C_1\cap C_2$, that is, the sequence $\{w^k\}_{k=0}^{\infty}$ is globally convergent to a solution of the feasibility problem \eqref{eqn:AVE_as_a_FP2}. 
\end{theorem}

\begin{proof}
	To prove this result, denote $\bar{w}^k \in P_{C_2}(w^k)$ for all $k\geq 0$. By \eqref{eqn:relaxedMAP}, we have 
	\begin{equation}\label{eqn:relaxedMAP2}
	w^{k+1} = (1-\gamma)\bar{w}^k + \gamma P_{C_1}(\bar{w}^k)
	\end{equation}
	and so
	\begin{equation}\label{eqn:relaxedMAPswitched}
	\bar{w}^{k+1} \in P_{C_2}(w^{k+1}) = P_{C_2}((1-\gamma)\bar{w}^k + \gamma P_{C_1}(\bar{w}^k)).
	\end{equation}
	Let $h(w) := \frac{1}{2}\| w - P_{C_1}(w)\|^2$. Then $h$ is a Lipschitz continuous function with Lipschitz constant $1$ and $\nabla h(w) = w - P_{C_1}(w)$. Thus, $w - \gamma \nabla h(w) = (1-\gamma) w + P_{C_1}(w)$. In turn, \eqref{eqn:relaxedMAPswitched} reduces to $\bar{w}^{k+1} \in P_{C_2}(\bar{w}^k - \gamma \nabla h(\bar{w}^k))$. From~\cite[Theorem~5.3]{ABS13}, we conclude that there exists a point $\bar{w}^* \in C_2$ such that $\bar{w}^k \to \bar{w}^*$ and 
	\begin{equation}\label{eqn:stationarypoint}
	0\in \nabla h (\bar{w}^*) + N_{C_2}(\bar{w}^*).
	\end{equation}
	By continuity of $P_{C_1}$ and using equation \eqref{eqn:relaxedMAP2}, we see that $w^k \to w^* = (1-\gamma)\bar{w}^* + \gamma P_{C_1}(\bar{w}^*)$, which proves the first claim. 
	
	To prove the last claim, note that from \eqref{eqn:stationarypoint}, there exists $z^* \in N_{C_2}(\bar{w}^*)$ such that $z^* = P_{C_1}(\bar{w}^*)-\bar{w}^*$. The latter equation implies that $z^* \in \Ker (T)^{\perp}$. Since $(\bar{u}_i^*,\bar{v}_i^*)\neq (0,0)$ for all $i=1,\dots, n$, it follows from \eqref{eqn:NormalconeC2} that $N_{C_2}(\bar{w}^*) \subset \hat{C}_2$. By condition \eqref{eqn:linearlyregularintersection}, we conclude that $z^* = 0$ and therefore $\bar{w}^* = P_{C_1}(\bar{w}^*)$, i.e. $\bar{w}^* \in C_1$. Hence, $w^* = \bar{w}^* \in C_1\cap C_2$. 
\end{proof}

\medskip 
Notice that the above theorem guarantees linear convergence of \eqref{eqn:MAP_2} to a point in $C_1\cap C_2$, which may not be the same as the point $w^*$. On the other hand, Theorem \ref{theorem:linearrate} shows that linear convergence to $w^*$ is achieved. 

\subsection{A related fixed point algorithm}\label{subsec:relatedfixedpointalgorithm}
Another algorithm can also be derived from the method of alternating projections. To describe this algorithm, we denote by $D$ the multivalued mapping from $\Re^n\times \Re^n$ to $\Re^{2n\times 2n}$ such that for each $w=(u,v)\in \Re^n\times \Re^n$, $D(w)$ is a set containing diagonal matrices $D_w$ such that
\begin{equation*}\label{eqn:D(w)}
((D_w)_{ii},(D_w)_{n+i,n+i}) \in \begin{cases}
\{ (1,0)\} & \text{if}~ u_i>v_i ,~u_i \geq 0 \\
\{ (0,1)\} & \text{if}~u_i<v_i, ~ v_i\geq 0 \\
\{ (0,1), (1,0)\} & \text{if}~u_i=v_i> 0 \\
\{ (0,0)\} & \text{if}~u_i=v_i\leq 0
\end{cases},
\end{equation*}
for all $i=1,\dots, n$. Then, the projection onto $C_2$ can equivalently written as 
\[ P_{C_2}(w)  = D(w)w  = \{ D_w w : D_w\in D(w)\}. \]
Suppose now that $w$ is a fixed point of $P_{C_1}\circ P_{C_2}$, i.e. $w\in (P_{C_1} \circ P_{C_2})(w)$. Recalling that $P_{C_1}(w) = Lw + \sqrt{2}T^{\dagger}c$ where $L = I_{2n} - T^{\dagger}T$, we have
\[w = LD_w w + \sqrt{2}T^{\dagger}c, \qquad \text{where}~D_w \in D(w).\]
That is, $w\in \Fix (P_{C_1}\circ P_{C_2}) $ if and only if there exists $D_w \in D(w)$ such that 
\begin{equation*}\label{eqn:fixedpointequation3}
(I_{2n}-LD_w)w = \sqrt{2}T^{\dagger}c .
\end{equation*}
This motivates the iterations
\begin{equation}\label{eqn:fixedpointalgorithm}
w^{k+1} = \sqrt{2}(I_{2n}-LD_{w^k})^{-1}T^{\dagger}c,
\end{equation}
where $D_{w^k}\in D(w^k)$. This algorithm is well-defined if $1$ is not an eigenvalue of $LD(w^k)$ for all $k$. A particular case is described in the following proposition.  

\begin{proposition}
	The iterations \eqref{eqn:fixedpointalgorithm} are well-defined for any initial point $w^0\in \Re^n \times \Re^n$ if $Q$ given by \eqref{eqn:Qmatrix} is nondegenerate.
\end{proposition}

\begin{proof}
	We show that for all $w\in \Re^n \times \Re^n$, the matrix $I_{2n}-LD_w$ is nonsingular for any $D_w \in D(w)$. To this end, let $\Lambda = \{ i\in \{ 1,\dots, 2n\} ~:~ (D_w)_{ii} = 1\}$. Then 
	\[\| LD_w \| = \| L_{\cdot \Lambda} \| \leq \| L\| ,\]
	where the inequality follows from Lemma \ref{lemma:singularvalue_submatrix}. Since $Q$ is nondegenerate, $\| L\| <1 $ by Lemma~3.9. Thus, $\| LD_w \| < 1$ and therefore $I_{2n}-LD_w$ is nonsingular, as desired. 
\end{proof}

\medskip 
Unlike the MAP iterates \eqref{eqn:MAP_2}, it is not difficult to show that any sequence generated via \eqref{eqn:fixedpointalgorithm} is bounded. 

\begin{proposition}
	Let $\{w^k\}_{k=0}^{\infty}$ be any sequence generated by \eqref{eqn:fixedpointalgorithm}. Then $\{w^k\}_{k=0}^{\infty}$ is a bounded sequence. Any accumulation point $w^*$ of $\{w^k\}_{k=0}^{\infty}$ satisfies $w^* = \sqrt{2}(I_{2n} - LD^*)^{-1}T^{\dagger}c$, where $D^*\in \Re^{2n\times 2n}$ is a diagonal matrix with diagonal elements of $1$ or $0$ and satisfy $D_{ii}^*D_{n+i,n+i}^* = 0$ for all $i=1,\dots, n$. 
\end{proposition}

\begin{proof}
	Note that the range of the multivalued mapping $D$ is a finite set. In particular, the set $\{D_w : D_w \in D(w) ~\text{and}~w\in \Re^n\times \Re^n\}$ has $3^n$ elements. Thus, there exists a constant $\kappa \in (0,\infty)$ such that $\| (I_{2n} - LD_w)^{-1}\| \leq \kappa $ for all $w\in \Re^n\times \Re^n$ and $D_w \in D(w)$. Thus, $\|w^{k+1}\| \leq \sqrt{2} \|(I_{2n} - LD_{w^k})^{-1}\|  \cdot  \|T^{\dagger}c\| \leq \sqrt{2}\kappa \|T^{\dagger}c\|$ for all $k$. Hence, $\{w^k\}_{k=0}^{\infty}$ is a bounded sequence. To prove the last claim, let $w^*$ be an arbitrary accumulation point of $\{w^k\}_{k=0}^{\infty}$, and let $\{ w^{k_j}\}_{j=1}^{\infty}$ be a subsequence that converges to $w^*$. Denote by $d^{k_j - 1}$ the diagonal entries of $D_{w^{k_j - 1}}$. Then the sequence $\{ (w^{k_j},d^{k_j - 1})\}_{j=1}^{\infty}$ is bounded and must have subsequence that converges to some point $(w^*,d^*)$. Without loss of generality, we may assume that $\{ (w^{k_j},d^{k_j - 1})\}_{j=1}^{\infty}$ converges to $(w^*,d^*)$. It follows that $D_{w^{k_j - 1}}\to D^*$ as $j\to\infty$, where $D^*$ is the diagonal matrix with diagonal entries equal to $d^*$. Setting $k=k_j$ in \eqref{eqn:fixedpointalgorithm} and letting $j\to \infty$, we get the desired conclusion. 
\end{proof}
\medskip 
Both the MAP algorithm \eqref{eqn:MAP_2} and the iterations \eqref{eqn:fixedpointalgorithm} are aimed at finding a fixed point of $P_{C_1}\circ P_{C_2}$. However, the iterations \eqref{eqn:fixedpointalgorithm} require more computational effort than MAP since the former involves solving a linear system involving $2n$ equations in $2n$ unknowns for each iteration. Nevertheless, we may consider a hybrid algorithm where we generate first a sequence of MAP iterates, then use \eqref{eqn:fixedpointalgorithm} for the succeeding iterations. We call this approach the MAP-LS algorithm (where LS denotes linear system involved in computing the iterations given by \eqref{eqn:fixedpointalgorithm}) which is described in Algorithm \ref{algorithm}. Whenever convergent, the limit of the sequence generated by MAP-LS algorithm is necessarily a fixed point of $P_{C_1} \circ P_{C_2}$. 

\medskip

\begin{algorithm}[H]\label{algorithm}
	\SetAlgoLined 
	Choose a termination parameter $\varepsilon$ and set $w^0 = T^{\dagger}c$. Let $N$ be a positive integer and $\delta>0$. Set $k=0$. \\
	\begin{description}
		\item[Step 1.] Let 
		\[w^{k+1} \in  \begin{cases}
		(P_{C_1} \circ P_{C_2})(w^k) & \text{if}~k\leq N ~\text{and}~\|w^{k+1}-w^k\| >  \delta\\
		\sqrt{2}(I_{2n}-LD_{w^k})^{-1}T^{\dagger}c & \text{if}~k > N~\text{~or}~\|w^{k+1}-w^k\| \leq \delta.
		\end{cases} .\] 
		\item[Step 2.] Set $x^{k+1} = \frac{1}{\sqrt{2}}(u^{k+1} - v^{k+1})$. 
		\item[Step 3.] Stop if $\|Ax^{k+1} + B|x^{k+1}|-c\| \leq \varepsilon$. Otherwise, set $k=k+1$ and go to Step 1. 
	\end{description}
	\caption{MAP-LS algorithm}
\end{algorithm}

\section{Numerical simulations}\label{sec:numerical}
In this section, we demonstrate the applicability of MAP and MAP-LS in solving randomly generated absolute value equations \eqref{eqn:AVE}. We first note some remarks on the implementation of our algorithms.

\subsection{Implementation of MAP and MAP-LS}\label{sec:implementation}
If $T = [\begin{array}{ll}
A+B & -A+B
\end{array}]\in \Re^{m\times 2n}$ is of full row rank, then its Moore-Penrose inverse of $T$ is well-known and is given by 
\begin{equation*}
T^{\dagger} = T^{\T} (TT^{\T})^{-1} .
\end{equation*}
In view of Proposition~\ref{prop:P_C1}, we calculate the projection onto $C_1$ of a point $w\in \Re^n\times \Re^n$ by first solving for $z$ in  
\begin{equation}\label{eqn:linearsystem}
TT^{\T} z = Tw - \sqrt{2}c,
\end{equation}
then setting $P_{C_1}(w) = w - T^{\T}z$.
\medskip 

Notice that since $T$ is of full row rank, the coefficient matrix $TT^{\T}$ of the linear system \eqref{eqn:linearsystem} is a symmetric positive definite matrix, so we can use its Cholesky decomposition. In particular, we use the Matlab function \texttt{dS = decomposition(S,'chol')} where $S:= TT^{\T}=2(AA^{\T} + BB^{\T})$ and solve for $z$ in \eqref{eqn:linearsystem} by using the backslash operator, i.e. \texttt{z = dS\textbackslash b} where $b := Tw-\sqrt{2}c$. 

\medskip 
In particular, by virtue of Lemma~\ref{lemma:pseudoinverse_property}, the above procedure can be applied when dealing with the traditional AVE \eqref{eqn:AVE} with $A\in \Re^{n\times n}$ and $B=-I_n$. Furthermore, in this case, the matrix-vector multiplication $T^{\T}z$ can be calculated more efficiently by computing first $z':= A^{\T}z$ so that $T^{\T}z = (z'-z, -z'-z)$.  

\medskip 
On the other hand, the inversion of $2n\times 2n$ matrix in equation \eqref{eqn:fixedpointalgorithm} may be computationally intensive. However, since $I_{2n}-LD_{w^k}$ can be partitioned into four $n\times n$ blocks, then its inverse can be calculated in terms of the inverses of two $n\times n$ matrices. Particularly, 
if we let $L = \left[ \begin{array}{cc}
L_1 & L_2 \\ L_2^{\T} & L_3
\end{array}\right]$ and $D_{w^k}= \left[ \begin{array}{cc}
D_1^k & 0 \\ 0 & D_2^k
\end{array}\right]$, then
\[I_{2n}-LD_{w^k} = \left[ \begin{array}{cc}
I_n - L_1 D_1^k  & -L_2 D_2 ^k \\ -L_2^TD_1^k  & I_n - L_3D_2^k
\end{array}\right]. \] 
Thus, the inverse of $I_{2n}-LD_{w^k}$ can be calculated in terms of the inverse of $I_n - L_1 D_1^k$ and the inverse of its Schur complement (or the inverse of $I_n - L_3 D_2^k$ and the inverse of its Schur complement). These inverses exist, in particular, if $\|LD_{w^k}\|<1$ (such as when $Q$ is nondegenerate) in which case $\|LD_1^k\|<1$ and $\|LD_2^k\|<1$ by Lemma~\ref{lemma:singularvalue_submatrix}. In general, such approach is more efficient than dealing directly with the inverse of $I_{2n}-LD_{w^k}$. Hence, we take this approach when using the MAP-LS algorithm. 

\subsection{Numerical results}
We compare MAP and MAP-LS to four other algorithms in the literature, each of which is a representative of the four classifications described in the introduction. We only choose those algorithms which, like MAP and MAP-LS, do not require parameters which need to be tuned carefully. From the class of algorithms based on Newton methods, we choose the generalized Newton method (GNM)~\cite{Man08} as the other variants of the Newton method involve parameters that may be problem-dependent or are difficult to tune. From the second group, we choose the Picard iteration method (PIM) in~\cite{RHF14}. The variant of this method presented in~\cite{Salkuyeh14} is only applicable for positive definite matrices and involves a problem-dependent parameter. On the other hand, the iterates of the Douglas-Rachford splitting method~\cite{CYH21} are simply convex combinations of the PIM iterates and the current iterate (similar to the MAP relaxation \eqref{eqn:relaxedMAP}). In fact, if we use the prescribed parameters in~\cite{CYH21}, the Douglas-Rachford iterates approximate the PIM iterates. From matrix splitting iteration methods, we choose the Gauss-Seidel iteration~\cite{EHS17}. The SOR-like iteration method~\cite{KM17} also requires a parameter, and from the numerical results presented in~\cite{KM17}, we see that the SOR-like iteration also generates iterates which are approximately the same as the PIM iterations for optimally chosen parameters. Finally, we note that the concave minimization approach involves solving a linear program at each iteration, which may be inefficient for large scale problems. We omit comparisons with this approach for the case $B=-I_n$ as the current algorithms in the literature~\cite{Man07-2,Man15} are not competitive enough with the other methods. However, we use the successive linearization algorithm (SLA) in~\cite{Man07} for the general AVE \eqref{eqn:AVE}, which is the only existing algorithm in the literature that can solve such problems. 

\medskip

We briefly describe the algorithms we have chosen for our numerical comparisons:

\begin{itemize}
	\item[(a)] \textit{Generalized Newton method (GNM)}~\cite{Man08}\\ 
	This algorithm is aimed at solving the AVE \eqref{eqn:AVE} with $m=n$ and $B=-I_n$, and the iterations are given by 
	\begin{equation}\label{eqn:GNM}
	x^{k+1} = (A-D^k)^{-1}c ,
	\end{equation}
	where $D^k = \diag (\sgn(x_1^k),\dots, \sgn (x_n^k))$. The iterations are derived by applying the semismooth Newton method in solving the equation $Ax-|x| - c = 0$. As in~\cite{Man08}, we use the Matlab's backslash operator ``$\backslash$'' to obtain the iterates. The maximum iterations for this algorithm is set to 2000. 
	\item[(b)] \textit{Picard iteration method (PIM)}~\cite{RHF14} \\
	This method is applicable whenever $m=n$ and $A$ is invertible. The algorithm consists of the fixed point iterations for the equation $x=A^{-1}(-B|x|+c)$, that is, 
	\begin{equation}\label{eqn:Picard}
	x^{k+1} = A^{-1} (-B|x^k|+c).
	\end{equation}
	From the above formula, we only need to compute $A^{-1}$ once. For the sake of efficiency, we pre-compute the LU decomposition of $A$ using the \texttt{decomposition} function of Matlab. We set the maximum number of iterations to 2000.
	
	\item[(c)] \textit{Gauss-Seidel iteration method (GSM)}~\cite{EHS17} \\
	Similar to the generalized Newton method, this Gauss-Seidel algorithm solves the AVE $Ax-|x|=c$ by decomposing $A$ as $A=D-E-F$ where $D$, $E$ and $F$ are diagonal, strictly lower triangular and strictly upper triangular matrices. Using this decomposition, the Gauss-Seidel iterations are given by
	\[(D-E)x^{k+1} - |x^{k+1}| = Fx^k + c.\]
	Though the above system is nonlinear, the next iterate $x^{k+1}$ can be easily solved since $D-E$ is lower triangular. In particular, having computed $x_1^{k+1}$, we inductively compute $x_i^{k+1}$ using the previously obtained coordinates $x_1^{k+1}, x_2^{k+1}, \dots, x_{i-1}^{k+1}$. We set the maximum iterations to 10000.
	\item[(d)] \textit{Successive linearization algorithm (SLA)}~\cite{Man07} \\ 
	This is the only algorithm in the existing literature which can handle the general AVE \eqref{eqn:AVE}. Given an initial point $(x^0,t^0,s^0) \in \Re^n \times \Re^n \times \Re^m $, we solve the linear programming problem
	\begin{align*}
	\min & ~ \epsilon \sum _{i=1}^n (\sgn(x_i^k) x_i + t_i) + \sum _{j=1}^m s_i \\
	\text{s.t.} & ~ - s \leq Ax+Bt-c \leq s \\ 
	& ~ -t \leq x \leq t ,
	\end{align*}
	and call its solution $(x^{k+1}, t^{k+1}, s^{k+1})$. To solve this linear program, we use the Matlab function \texttt{linprog}. We set the maximum number of iterations to 1000. 
\end{itemize}

\medskip

All simulations were carried out in Matlab R2020a on a desktop machine with an Intel Core  i7-8700 3.20 GhZ and 32GB of memory. We use the zero vector as the initial point for all the algorithms, and the stopping criterion is
\begin{equation}\label{eqn:stoppingcriterion}
\| Ax^k + B|x^k| - c\| \leq \varepsilon \quad \text{with}~\quad \varepsilon = 10^{-6}.
\end{equation}	
\medskip 
For the case $m=n$ and $B=-I_n$, we compare our algorithms with GNM, PIM and GSM, since SLA takes a lot of computing time in solving the these problems. For the general case, we can only compare our algorithms with the SLA as the other solvers can only handle the case $m=n$. 

\begin{example}\label{example1}
	We generate a matrix $A$ as in~\cite{Man08}. First, we generate a matrix $A'\in \Re^{n\times n}$ whose entries are from the uniform distribution on $[-10,10]$. Then, we let $A = A'/(t\sigma_{\min}(A'))$ where $t$ is a uniform random number in $[0,1]$. We then randomly generate a vector $x^*\in \Re^n$ such that $x_i^* = r \cdot 10^{\alpha s}$ where $\alpha \in \{ 0,1,2,3\}$, while $r$ and $s$ are generated from the uniform distribution on $[-1,1]$ and $[0,1]$, respectively. Finally, we set $c=Ax+B|x|$, where $B=-I_n$. We note that the case $\alpha = 0 $ is precisely the test problem considered in~\cite{Man08}. 
	
	\medskip 
	In this example, $\sigma_{\min}(A) > \sigma_{\max}(B)$ so that the AVE \eqref{eqn:AVE} has a unique solution (see Remark~\ref{remark:AVE_equiv_LCP}). For our experiments, we let $n=5000$ and generate 100 random AVEs as described above. We report in Table \ref{table:Example1} the success rates and averages of CPU time and number of iterations (of successful simulations) of MAP, GNM, PIM and GSM. First, note that PIM has the best average CPU time in solving the AVEs, followed by GNM and our MAP algorithm. However, in terms of reaching a solution with residual given by \eqref{eqn:stoppingcriterion}, both GNM and PIM has relatively lower success rates compared to MAP. Moreover, GNM and PIM failed to solve several test problems as $\alpha$ increases. In particular, both of these algorithms failed to solve 100 randomly generated AVEs when $\alpha = 3$. On the other hand, our algorithm is still able to solve more than $60\%$ of the problems when $\alpha =3$. Finally, notice that Gauss-Seidel method failed to solve all the problems. For this algorithm, each component of the iterate $x^{k+1}$ is obtained by solving a nonlinear equation of the form $ax-|x|=b$. This equation might not have a solution for $b\neq 0$ if $b/(a-1)<0$ and $b/(a+1)>0$, which is the reason why GSM failed in solving the generated AVEs. In fact, this problem was encountered by GSM during the first iteration for all of the test problems considered. 
\end{example}

\medskip 

\begin{table}[htbp]
	\centering
	\caption{Numerical results for Example~\ref{example1}.}
	\begin{tabular}{|c|l|r|r|r|r|}
		\hline
		\multirow{2}[4]{*}{Method} & \multirow{2}[4]{*}{} & \multicolumn{4}{c|}{$\alpha$} \bigstrut\\
		\cline{3-6}          &       & \multicolumn{1}{c|}{0} & \multicolumn{1}{c|}{1} & \multicolumn{1}{c|}{2} & \multicolumn{1}{c|}{3} \bigstrut\\
		\hline
		\multirow{3}[2]{*}{MAP} & Success($\%$) & 1     & 0.99  & 0.87  & 0.62 \bigstrut[t]\\
		& Ave. Time & 2.58  & 3.03  & 3.13  & 10.42 \\
		& Ave. Iter & 40.85 & 52.51 & 55.44 & 250.39 \bigstrut[b]\\
		\hline
		\multicolumn{1}{|c|}{\multirow{3}[2]{*}{generalized Newton method}} & Success($\%$) & 0.76  & 0.55  & 0     & 0 \bigstrut[t]\\
		& Ave. Time & 2.23  & 2.29  & $-$   & $-$ \\
		& Ave. Iter & 3.93  & 4.00  & $-$   & $-$ \bigstrut[b]\\
		\hline
		\multirow{3}[2]{*}{Picard iteration method} & Success($\%$) & 0.75  & 0.54  & 0.01  & 0 \bigstrut[t]\\
		& Ave. Time & 0.57  & 0.59  & 0.84  & $-$ \\
		& Ave. Iter & 4.99  & 5.65  & 22.00 & $-$ \bigstrut[b]\\
		\hline
		\multicolumn{1}{|c|}{\multirow{3}[2]{*}{Gauss-Seidel iteration method}} & Success($\%$) & 0     & 0     & 0     & 0 \bigstrut[t]\\
		& Ave. Time & $-$   & $-$   & $-$   & $-$ \\
		& Ave. Iter & $-$   & $-$   & $-$   & $-$ \bigstrut[b]\\
		\hline
	\end{tabular}%
	\label{table:Example1}%
\end{table}%

\medskip 

\begin{example}\label{example2}
	We set $B=-I_n$ and let $A=(A')^{\T}A'$ where $A'\in \Re^{n\times n}$ is sampled from the standard normal distribution. We also randomly generate a vector $x^*$ from the standard normal distribution, and set $c = Ax^* + B|x^*|$. For each $n\in \{500, 1000,2000,3000\}$, we generate $100$ random AVEs as described and solve them using MAP-LS, GNM, PIM and GSM. The summary of the results is reported in Table \ref{table:Example2}. For the MAP-LS algorithm, we set $N=100$ and $\delta = 10^{-3}$ in Algorithm \ref{algorithm}. In Table \ref{table:Example2}, we also report two averages of iteration numbers for MAP-LS: (i) ``Ave. Iter (MAP)'' indicates the average number of MAP iterations \eqref{eqn:MAP_2} of successful instances, and (ii) ``Ave. Iter (LS)'' indicates the average number of the linear system iterations \eqref{eqn:fixedpointalgorithm} of successful simulations. 
	
	We see from Table \ref{table:Example2} that MAP-LS used $100$ iterations of the alternating projections \eqref{eqn:MAP_2} for all the test problems, before using the iterations \eqref{eqn:fixedpointalgorithm}. Moreover, the average number of iterations via \eqref{eqn:fixedpointalgorithm} increases as the dimension $n$ increases. Despite this, it is evident that the average CPU time required by MAP-LS to solve the AVEs is significantly shorter than the time required by GNM. In fact, the gap in CPU times spent by MAP-LS and GNM becomes more apparent as the dimension of the problem increases. This is due to the fact that GNM took much more iterations than MAP-LS. Recall that using the implementation described in Section \ref{sec:implementation}, each iteration of MAP-LS requires two $n\times n$ matrix inversions, while from \eqref{eqn:GNM}, we see that GNM only needs to invert a single $n\times n$ matrix at each iteration. However, as GNM took significantly more iterations than MAP-LS, the latter significantly outperforms the former. 
	
	In addition, MAP-LS achieved at least $75\%$ success rate in solving the AVEs of different dimensions $n$, while the success rate of GNM decreases dramatically as $n$ increases. In particular, for $n=3000$, GNM only achieved less than $25\%$ success rate. Finally, both PIM and GSM failed to solve all the generated test problems after reaching the maximum number of iterations set. 
\end{example}

\medskip 
\begin{table}[htbp]
	\centering
	\caption{Numerical Results for Example~\ref{example2}}
	\begin{tabular}{|c|l|r|r|r|r|}
		\hline
		\multirow{2}[4]{*}{Method} & \multirow{2}[4]{*}{} & \multicolumn{4}{c|}{$n$} \bigstrut\\
		\cline{3-6}          &       & \multicolumn{1}{c|}{500} & \multicolumn{1}{c|}{1000} & \multicolumn{1}{c|}{2000} & \multicolumn{1}{c|}{3000} \bigstrut\\
		\hline
		\multirow{4}[2]{*}{MAP-LS} & Success($\%$) & 0.78  & 0.81  & 0.76  & 0.76 \bigstrut[t]\\
		& Ave. Time & 0.21  & 1.40  & 10.81 & 36.67 \\
		& Ave. Iter (MAP) & 100   & 100   & 100   & 100 \\
		& Ave. Iter (LS) & 10.38 & 15.14 & 19.41 & 23.38 \bigstrut[b]\\
		\hline
		\multicolumn{1}{|c|}{\multirow{3}[2]{*}{generalized Newton method}} & Success($\%$) & 0.84  & 0.84  & 0.63  & 0.22 \bigstrut[t]\\
		& Ave. Time & 0.51  & 6.28  & 74.91 & 252.69 \\
		& Ave. Iter & 118.55 & 335.90 & 803.70 & 1067.32 \bigstrut[b]\\
		\hline
		\multirow{3}[2]{*}{Picard iteration method} & Success($\%$) & 0     & 0     & 0     & 0 \bigstrut[t]\\
		& Ave. Time & $-$   & $-$   & $-$   & $-$ \\
		& Ave. Iter & $-$   & $-$   & $-$   & $-$ \bigstrut[b]\\
		\hline
		\multicolumn{1}{|c|}{\multirow{3}[2]{*}{Gauss-Seidel iteration method}} & Success($\%$) & 0     & 0     & 0     & 0 \bigstrut[t]\\
		& Ave. Time & $-$   & $-$   & $-$   & $-$ \\
		& Ave. Iter & $-$   & $-$   & $-$   & $-$ \bigstrut[b]\\
		\hline
	\end{tabular}%
	\label{table:Example2}%
\end{table}%

\begin{example}\label{example3}
	We sample the entries of $A, B\in \Re^{m\times n}$ and $x^*\in \Re^n$ from the standard normal distribution, and we set $c = Ax^*+B|x^*|$. We let $n=500$ and for each $m=rn$ with $r\in \{0.25, 0.5, 0.75, 1.5, 2, 3 \}$, we generate 100 random AVEs and solve these problems using MAP and SLA. The results are summarized in Table \ref{table:Example3}. Observe that both algorithms were able to solve all the randomly generated problems. However, it is noticeable that the difference in the average CPU time spent in solving the test problems is quite significant. More specifically, the ratios of the average CPU time of SLA to the average CPU time of MAP for the six values of $r$ considered are 516.60, 712.23,	247.15,	211.84,	1604.34 and 470.73, respectively. This shows the substantial difference in performance of the two algorithms. 
\end{example}

\begin{table}[htbp]
	\centering
	\caption{Numerical Results for Example~\ref{example3}}
	\begin{tabular}{|c|l|r|r|r|r|r|r|}
		\hline
		\multirow{2}[4]{*}{Method} & \multirow{2}[4]{*}{} & \multicolumn{6}{c|}{$r$} \bigstrut\\
		\cline{3-8}          &       & \multicolumn{1}{c|}{0.25} & \multicolumn{1}{c|}{0.5} & \multicolumn{1}{c|}{0.75} & \multicolumn{1}{c|}{1.5} & \multicolumn{1}{c|}{2} & \multicolumn{1}{c|}{3} \bigstrut\\
		\hline
		\multirow{3}[2]{*}{MAP} & Success($\%$) & 1     & 1     & 1     & 1     & 1     & 1 \bigstrut[t]\\
		& Ave. Time & 0.01  & 0.03  & 0.26  & 0.12  & 0.02  & 0.19 \\
		& Ave. Iter & 104.19 & 296.34 & 2162.84 & 227.16 & 1     & 1 \bigstrut[b]\\
		\hline
		\multicolumn{1}{|c|}{\multirow{3}[2]{*}{SLA}} & Success($\%$) & 1     & 1     & 1     & 1     & 1     & 1 \bigstrut[t]\\
		& Ave. Time & 4.21  & 19.69 & 63.60 & 26.11 & 31.33 & 90.31 \\
		& Ave. Iter & 2.38  & 3.64  & 6.11  & 1     & 1     & 1 \bigstrut[b]\\
		\hline
	\end{tabular}%
	\label{table:Example3}%
\end{table}%

\section*{Acknowledgements}
The first and second authors' research is supported by Ministry of Science and Technology, Taiwan. The third author's research is supported in part by DE200100063 from the Australian Research Council.

\bibliographystyle{amsplain}
\bibliography{bibfile}

\end{document}